\documentclass{article}

\usepackage{arxiv}

\usepackage{amsfonts}
\usepackage{mathrsfs}
\usepackage{amsmath}
\usepackage{amsthm}
\usepackage{amssymb}
\usepackage{setspace}
\usepackage{bbm}
\usepackage{dsfont}
\usepackage{enumerate}
\usepackage{mathtools}
\usepackage{hyperref}
\usepackage[color,matrix,arrow]{xy}
\usepackage[numbers,sort&compress]{natbib}

\usepackage[utf8]{inputenc} 
\usepackage[T1]{fontenc}    
\usepackage{hyperref}       
\usepackage{url}            
\usepackage{booktabs}       
\usepackage{amsfonts}       
\usepackage{nicefrac}       
\usepackage{microtype}      
\usepackage{lipsum}		

\numberwithin{equation}{section}

\theoremstyle{definition}
\newtheorem{thm}{Theorem}[section]
\newtheorem{lem}[thm]{Lemma}
\newtheorem{prop}[thm]{Proposition}
\newtheorem{cor}[thm]{Corollary}

\theoremstyle{definition}
\newtheorem{defn}{Definition}[section]

\newcommand{\R}{\mathbb{R}}
\newcommand{\N}{\mathbb{N}}

\newcommand{\Pb}{\mathbb{Pb}}
\newcommand{\E}{\mathbb{E}}
\newcommand{\C}{\mathbb{C}}
\newcommand{\HS}{\mathcal{H}}

\theoremstyle{remark}
\newtheorem*{rem}{Remark}

\title{On Azuma-type inequalities for Banach space-valued martingales}

\author{
  Sijie Luo\\
  Yau Mathematical Sciences Center \\
  Tsinghua University\\
  \texttt{luosijie@mail.tsinghua.edu.cn} \\
}

\begin{document}
\maketitle

\begin{abstract}
In this paper, we will study concentration inequalities for Banach space-valued martingales. Firstly, we prove that a Banach space $X$ is linearly isomorphic to a $p$-uniformly smooth space ($1<p\leq 2$) if and only if an Azuma-type inequality holds for $X$-valued martingales. This can be viewed as a generalization of Pinelis' work on Azuma inequality for martingales with values in $2$-uniformly smooth space. Secondly, Azuma-type inequality for self-normalized sums will be presented. Finally, some further inequalities for Banach space-valued martingales, such as moment inequalities for double indexed dyadic martingales and the De la Pe\~{n}a-type inequalities for conditionally symmetric martingales, will also be discussed.
\end{abstract}

\section{Introduction}
Concentration inequalities for sums of independent random variables and their extension to martingales have been extensively studied by many authors, such as Bernstein (1927), Kolomogrov (1929), Bennett (1962), Hoeffding (1963), Azuma (1967) etc. To state one, we begin with the following Hoeffding-Azuma inequality, which was proved by Hoeffding \cite{Ho} for sums of
independent bounded random variables and was extended to martingales by Azuma \cite{Az}.

\begin{thm}[Hoeffding-Azuma]\label{Azuma inequality}
Let $(f_{j})_{j=0}^{n}$ be a real-valued martingale such that $|f_{j}-f_{j-1}|\leq a_{j}~a.s.$ for all $j=1,2,\cdots n$. Then, for all $r\geq 0$,
\begin{equation}\nonumber
\Pb\{|f_{n}-f_{0}|\geq r\}\leq 2\exp\left\{-\frac{r^{2}}{2(a_{1}^{2}+\cdots+a_{n}^{2})}\right\}.
\end{equation}
\end{thm}

The Hoeffding-Azuma inequality now becomes a basic methodology for proving concentration inequalities for martingales with bounded differences. For some interesting extensions of the method of bounded differences we refer to the survey paper by Godbole and Hitczenko \cite{G-H}. On the other hand, the Hoeffding-Azuma inequality has been used as a powerful tool in various areas such as probability theory, graph theory, information theory, computer science and related fields (see e.g. \cite{{C-L},{Ha},{Mc},{R-S}}).

Recently, motivated by the study of random matrices, concentration inequalities for sums of independent matrix-valued random variables and their extension to matrix-valued martingales have attracted more and more attention. Numerous concentration inequalities have been extended to random variables with values in the space of matrices equipped with operator norm. Several work in this direction have been done by Oliveira \cite{Ol} Ahlswede and Winter \cite{A-W}, Tropp \cite{Tr1} etc. With the help of the non-commutative generalizations of classical concentration inequalities, fundamental theorems of random matrices have been established successfully. In the matrix-valued setting, deep theory from operator theory, such as the Lieb's concavity theorem \cite{Li} and the Golden-Thompson inequality, will be used to derived concentration inequalities. For further information on matrix-valued concentration inequalities we refer to \cite{Tr2} and the extensive references therein.

Motivated by the study of probability in Banach spaces, it is natural to consider concentration inequalities for Banach space-valued random variables. However, while dealing inequalities in this setting, the main obstacle for proving such inequalities is that the moment generating function and the cumulant generating function methods are not available. To overcome the difficulty, some new techniques must be needed and it will be no surprised that the techniques are not from probabilistic but come from analysis. A great deal of effort has been expended on the study of moment inequalities and integrability of sums of independent Banach space-valued random variables and remarkable results have been established by Ledoux, Talagrand, Maurey, Pisier and many others (see e.g.\cite{{L-T},{Pis2}}). However, to the best of our knowledge, the study of concentration inequalities for Banach space-valued random variables is not so widely explored.

In his fundamental work, Naor \cite[Theorem 1.5]{Na} established the following Azuma inequality for martingales with values in $p$-uniformly smooth Banach space ($1<p\leq 2$) and then he applied the inequality to obtain an improved estimate of the Alon-Roichman theorem for random Cayley graphs of Abelian groups. Before stating the Naor's result precisely, we shall now recall the \emph{$p$-uniform smoothness} of Banach spaces.

For a Banach space $X$, the following quantity is called the {\emph {modulus of uniform smoothness}} of $X$, for $\tau>0$
\begin{equation}\label{uniform smoothness}
\rho_{X}(\tau)=\sup\left\{\frac{\|x+\tau y\|+\|x-\tau y\|}{2}-1:x,~y\in X,~\|x\|=\|y\|=1  \right\}.
\end{equation}
Then $X$ is said to be \emph{uniformly smooth} if $\lim\limits_{\tau\to 0^{+}}\frac{\rho_{X}(\tau)}{\tau}=0$ and if moreover there exists a constant $s>0$ such that for all $\tau>0$ we have $\rho_{X}(\tau)\leq s \tau^{p}$, then $X$ is said to be \emph{$p$-uniformly smooth}. It is clear that the $p$-uniform smoothness are only available for $p\in(1,2]$ and, by the parallelogram law (i.e. $|x+y|^{2}+|x-y|^{2}=2(|x|^{2}+|y|^{2})$ for all $x,~y\in \HS$), Hilbert space $(\HS,|\cdot|)$ is of $2$-uniformly smooth.

The following theorem, due to Naor, is an Azuma-type inequality for Banach space-valued martingales.
\begin{thm}(\cite{Na})\label{Naor's Azuma inequality}
There exists a universal constant $c\in(0, \infty)$ with the following property. Fix $s>0$ and assume that a Banach space $(X,\|\cdot\|)$ satisfies $\rho_{X}(\tau)\leq s\tau^{2}$ for all $\tau>0$. Fix also a sequence of positive numbers $(a_{j})_{j=1}^{\infty}\subseteq (0,\infty)$. Let $(f_{j})_{j=1}^{\infty}$ be an $X$-valued martingale satisfying $\|f_{j}-f_{j-1}\|\leq a_{j}~a.s.$ for all $j\in \N$. Then, for every $r\geq 0$ and $n\in \N$, we have

\begin{equation}\label{Naor's inequality}
\Pb\{\|f_{n}-f_{0}\|\geq r\}\leq e^{s+2}\cdot\exp\left\{-\frac{cr^{2}}{a_{1}^{2}+\cdots+a_{n}^{2}} \right\}.
\end{equation}
\end{thm}
As stated in \cite[p. 626]{Na}, Naor further pointed out: ``\emph{ All our results carry over, with obvious modifications, to general uniformly smooth spaces (of course, in this more general setting, the probabilistic bounds that we get will no longer sub-Gaussian)}''. Indeed, the main ingredient for proving \eqref{Naor's inequality} is the \cite[Lemma 2.2]{Na}, which sates that if a $2$-uniformly smooth space $X$ with the modulus of uniform smoothness satisfies that $\rho_{X}(\tau)\leq s \tau^{2}$ for all $\tau>0$, then, for every $x$, $y\in X$ and every $q\geq 2$ we have
\begin{equation}\label{lemma for 2-uniformly smooth spaces}
\frac{\|x+y\|^{q}+\|x-y\|^{q}}{2}\leq (\|x\|^{2}+8(s+q)\|y\|^{2})^{q/2}.
\end{equation}
The proof of this lemma can be generalized to $p$-uniformly smooth spaces by replacing $2$ of \eqref{lemma for 2-uniformly smooth spaces} with $p$ for all $q\geq p$. Combining with this modification and the Naor's proof of Theorem \ref{Naor's Azuma inequality}, one could deduce the corresponding Azuma-type inequality for $p$-uniform smooth space-valued martingales ($1<p\leq 2$) by replacing $r^{2}$ and $a_{1}^{2}+\cdots+a_{n}^{2}$ on the right hand side of \eqref{Naor's inequality} by $r^{p}$ and $a_{1}^{p}+\cdots+a_{n}^{p}$ respectively.

However, constants in \eqref{Naor's inequality} are worse than those in \eqref{Azuma inequality} and it seems that Naor's method could not lead to better constants even for Hilbert space-valued martingales or certain sub-classes of martingales. Hence, a natural question we consider most is about providing a more precise version of the Azuma-type inequality for martingales with values in $p$-uniformly smooth space $X$, such that, when the Banach space $X$ is a Hilbert space, our results recover the classical Azuma inequality. On the other hand, it is also worthwhile to ask that whether the $p$-uniform smoothness is necessary for Azuma-type inequality to hold. All these questions are natural and worth considering, which could be stated as follows.
\begin{itemize}
\item Can we provide an Azuma-type inequality for Banach space-valued martingales with some improvements?
\item Is the $p$-uniform smoothness of the image space necessary for Azuma-type inequality?
\end{itemize}

To answer these questions we found that some improvements of Naor's work must be needed. However, as it will be seen below, the Naor's proof is based on some geometric arguments from Banach space theory and estimation of an implicity constant in Pisier's inequality \cite{Pis1}, both of which are hard to be improved. Precisely, Naor proved that for martingale $(f_{j})_{j=0}^{\infty}$ taking values in $2$-uniformly smooth space $X$, then there exists a constant $s>0$ only depending on $X$ such that for every $q\geq 2$ and $n\in \N$ the following inequality holds

\begin{equation}\label{Naor's lemma}
\big{(}\E\|f_{n}-f_{0}\|^{q}\big{)}^{\frac{1}{q}}\leq 8\sqrt{s+q}\cdot\large\sqrt{\sum\limits_{j=1}^{n}(\E\|d_{j,f}\|^{q})^{\frac{2}{q}}},
\end{equation}
where $d_{j,f}=f_{j}-f_{j-1}$ for $j\in \N$. Then, the optimal choice of $q\geq 2$ in \eqref{Naor's lemma} will yield the desired inequality. However, in order to make optimal choice of $q\geq 2$ of \eqref{Naor's lemma} to derive Azuma inequality, one needs to involve the $\ell_{\infty}$-norm of the martingale differences $d_{j,f}$. And it seems that even for certain sub-classes of martingales (such as conditionally symmetric martingales) the \eqref{Naor's lemma} could not lead to a better inequality.

To tackle the problem, our approach is based on a method of Pinelis, who have established some fundamental inequalities for $2$-uniformly smooth space-valued martingales (see \cite{{Pin1},{Pin2}}) and then we apply a ``martingale dimension reduction'' argument (inspired by \cite[Lemma 4.2]{D-L-P}) to reduce martingales from $p$-uniformly smooth space-valued to $\R^{2}$-valued. More precisely, we prove the following theorem (corresponds to Corollary \ref{refinement of Azuma inequality})
\begin{thm}
Let $X$ be a $p$-uniformly smooth space ($1<p\leq 2$) and $f=(f_{j})_{j=0}^{\infty}$ be an $X$-valued martingale such that there exists a non-negative predictable sequence $w=(w_{j})_{j=1}^{\infty}$ satisfying $\|f_{j}-f_{j-1}\|\leq w_{j}~a.s.$ for all $j\in \N$. Then, there exists a constant $K$ only depending on $X$ such that for all $r\geq 0$,
\[
\Pb\{f^{*}\geq r\}\leq 2\exp\left\{-\frac{r^{p}}{2K\Big{\|}\sum_{j=1}^{\infty}w_{j}^{p}\Big{\|}_{\infty}} \right\},
\]
where $f^{*}=\sup\limits_{j\in \N}\|f_{j}-f_{0}\|$.
\end{thm}
This theorem is an improvement of Naor's Azuma inequality for martingales with values in $p$-uniformly smooth space ($1<p\leq 2$). By the method of ``\emph{good-$\lambda$ inequality}'' of Burkholder, we can further deduce that the $p$-uniform smoothness of the image space is necessary for Azuma-type inequalities to hold. Moreover, we found that some other types of inequalities, such as Azuma-type inequality for self-normalized sums, De la Pe\~{n}a inequalities etc. are also true for $p$-uniformly smooth space-valued martingales.

Our paper is organized as follows.

In Section 2 (the next section), preliminary results on martingales theory and geometric properties of uniformly smooth spaces will be collected.

In Section 3, we will focus on the Azuma-type inequalities for martingales with values in $p$-uniformly smooth space. For example, we prove that a martingale $f=(f_{j})_{j=0}^{\infty}$ with values in $p$-uniformly smooth space, then an Azuma-type inequality holds for $f=(f_{j})_{j=0}^{\infty}$. This can be viewed as an improvement of Naor's Azuma inequality \eqref{Naor's Azuma inequality} and a generalization of Pinelis' Azuma-type inequality for martingales with values in $2$-uniformly smooth space (see \cite{Pin1}). Furthermore, by applying the ``good-$\lambda$ inequality'', we deduce that the $p$-uniform smoothness of the image space is necessary for this type of inequality to hold. Hence, we obtain a characterization of $p$-uniformly smooth spaces in terms of Azuma-type inequalities. At the end of this section, Azuma-type inequality for self-normalized sums will also be discussed.

Section 4 will contain some further inequalities for Banach space-valued martingales. In particular, moment inequalities for Banach space-valued double dyadic martingales and the De la Pe\~{n}a-type inequalites for $p$-uniformly smooth space-valued conditionally symmetric martingales will be presented.

Throughout this paper, $X$ (resp. $\HS$) stands for separable Banach (resp. Hilbert) space and $X^*$ is the dual space of $X$. For constants $K_{1}$, $K_{2}$,  the notation $K_{1}=O( K_{2})$ we mean that $K_{1}\leq D K_{2}$ with some universal constant $D>0$. We assume that $\inf\emptyset=\infty$ and $\prod_{j\in \emptyset}u_{j}=1$ for convenience. For random variables $f$, $g$ on some probability space, we will simply use ``$f\leq g$'' to stand that ``$f\leq g$ almost surely''. For $p\geq 1$, the notation $(L^{p}(X),\|\cdot\|_{p})$ stands for the space of all $p$-Bochner integrable $X$-valued functions equipped with the $L^{p}$-norm, that is, $\|f\|_{p}=\big{(}\int \|f(\omega)\|^{p}d\Pb(\omega) \big{)}^{1/p}$ for all $f\in L^{p}(X)$. For a random process $f=(f_{j})_{j=0}^{\infty}$ we simply denote $\sigma(f_{0},\cdots,f_{n})$ the $\sigma$-algebra generated by $f_{0},\cdots,f_{n}$.

\section{Preliminaries}
This section will be divided into two parts. The first part includes basic concepts and notations from martingale theory. The second part is a digression to Banach space theory.

\subsection{Basic concepts and results from martingale theory}

The following concepts can be found in any books on probability theory (see e.g. \cite{Kl}). A sequence of sub-$\sigma$-algebras $(\mathscr{F}_{j})_{j=0}^{\infty}$ of a probability space $(\Omega,\mathscr{F},\Pb)$ is said to be a filtration if 
$\mathscr{F}_{j-1}\subseteq \mathscr{F}_{j}$ for all $j\in \N$. A sequence $(f_{j})_{j=0}^{\infty}$ is called adapted to $(\mathscr{F}_{j})_{j=0}^{\infty}$ if $f_{j}$ is $\mathscr{F}_{j}$-measurable for each $j\in\N\cup\{0\}\coloneqq \N_{0}$. A random variable $\tau:(\Omega,\mathscr{F},\Pb)\to \N\cup\{\infty\}$ is called a stopping time \big{(}relatively to the filtration $(\mathscr{F}_{j})_{j=0}^{\infty}$\big{)} if $\{\tau=n\}\in\mathscr{F}_{n}$ for each $n\in \N$. Throughout this paper we assume that $\mathscr{F}_{0}=\big{\{}\{\emptyset\}, \Omega \big{\}}$ for convenience.

A sequence of $X$-valued Bochner-integrable random process $(f_{j})_{j=0}^{\infty}$ is called a martingale relatively to the filtration $(\mathscr{F}_{j})_{j=0}^{\infty}$, if $(f_{j})_{j=0}^{\infty}$ is adapted relatively to the filtration $(\mathscr{F}_{j})_{j=0}^{\infty}$ such that $\E(f_{j}|\mathscr{F}_{s})=f_{s}$ for all $s\leq j$ and $j\in\N$. If it does not cause any confusion, we shall simply call $(f_{j})_{j=0}^{\infty}$ a martingale and write $\E_{j}=\E(\cdot|\mathscr{F}_{j})$ for short. If $\E(\|f_{j}\|^{p})<\infty$ for all $j\in \N$ ($p\geq 1$) then $(f_{j})_{j=0}^{\infty}$ is said to be a $L^{p}$-martingale.

The following theorem is well known fact in probability theory called the ``\emph{Optional Sampling Theorem}''(see \cite[p. 209]{Kl}). Recall that a real-valued adapted random process $f=(f_{j})_{j=0}^{\infty}$ is called  sub-martingale (resp. super-martingale) if $f_{s}\leq \E_{s}(f_{j})$ (resp. $f_{s}\geq \E_{s}(f_{j})$) for all $s\leq j$.
\begin{thm}[Optional Sampling Theorem]\label{optional sampling theorem}
Let $f=(f_{j})_{j=0}^{\infty}$ be super-martingale and $\sigma\leq\tau$ be stopping times.
\begin{enumerate}[i)]
\item Assume there exists $T\in\N$ such that $\tau\leq T~a.s.$. Then
\begin{equation}\nonumber
f_{\sigma}\geq \E(f_{\tau}|\mathscr{F}_{\sigma}),
\end{equation}
and if $f$ is a martingale then the equality holds.

\item If $f$ is nonnegative and $\tau< \infty~a.s.$, then we have $\E(f_{\tau})\leq \E(f_{0})<\infty$,
$\E(f_{\sigma})\leq\E(f_{0})<\infty$ and $f_{\sigma}\geq \E(f_{\tau}|\mathscr{F}_{\sigma})$.

\item Assume that, more generally, $f$ is only adapted and integrable. Then $f$ is a martingale if and only if $\E(f_{\tau})=\E(f_{0})$ for any bounded stopping time $\tau$.
\end{enumerate}
\end{thm}

For an $X$-valued martingale $f=(f_{j})_{j=0}^{\infty}$, the maximal function $f^*$ (resp. $f_{n}^{*}$) is defined by $f^*=\sup_{j\in\N}\|f_{j}-f_{0}\|$ (resp. $f_{n}^*=\max_{1\leq j\leq n} \|f_{j}-f_{0}\|$). Let $(d_{j,f})_{j=1}^{\infty}$ be the martingale differences of $f$, that is, $d_{j,f}=f_{j}-f_{j-1}$ for all $j\in\N$, then define $S_{p}(f)=(\sum_{j=1}^{\infty}\|d_{j,f}\|^{p})^{\frac{1}{p}}$ \big{(}resp. $S_{p,n}(f)=(\sum_{j=1}^{n}\|d_{j,f}\|^{p})^{\frac{1}{p}}$\big{)}.

Recall that a martingale $f$ is said to be \emph{conditionally symmetric} if and only if
\begin{equation}\nonumber
\E\varphi(d_{1,f},\cdots,d_{j-1,f},d_{j,f})=\E\varphi(d_{1,f},\cdots,d_{j-1,f},-d_{j,f}),
\end{equation}
for all bounded continuous function $\varphi:X^{j}\to \R$ and all $j\in\N$. For example, all dyadic martingales are conditionally symmetric.

We now conclude this subsection with an important property of conditionally symmetric martingales (see e.g. \cite[p. 53]{Bu2}).

\begin{prop}\label{Burkholder's replace filtration argument}
Let $X$ be a Banach space and $f$ be an $X$-valued conditionally symmetric martingale. Then $f$ is also a martingale relatively to the filtration $(\mathscr{G}_{j})_{j=0}^{\infty}$ defined by
\begin{equation}\nonumber
\mathscr{G}_{j}=\sigma(d_{0,f},\cdots,d_{j,f},\|d_{j+1,f}\|),~\forall j\in\N_{0}
\end{equation}
where $d_{0,f}\coloneqq f_{0}$ for convenience.
\end{prop}

\subsection{Digress to Banach space theory}
Recall that a Banach space $X$ is said to be \emph{$p$-uniformly smooth} if there exists $s>0$ such that
\begin{equation}\label{uniform smooth constant}
\rho_{X}(\tau)\leq s\tau^{p}
\end{equation}
for all $\tau>0$ and denote by $s_{p}(X)$ the infimal constant such that \eqref{uniform smooth constant} holds.
It is clear that the $p$-uniform smoothness implies the $q$-uniform smoothness for $q\leq p$, hence, in order to avoid confusion, we say that a Banach space is $p$-uniformly smooth where the $p$ refers to the best possible for which $\rho_{X}(\tau)\leq s\tau^{p}$ holds for each $\tau>0$.

For a proper extend real-valued lower semicontinuous convex function $h:X\to\R\cup\{+\infty\}$, the subdifferential of $h$ is a set-valued mapping $\partial h:X\to 2^{X^*}$ defined by, for $x\in X$,
\begin{equation}\nonumber
\partial h(x)=\{x^*\in X^*:\langle x^*,y-x \rangle\leq f(y)-f(x),~\forall y\in X\}.
\end{equation}
\begin{defn}
For $p>1$, the \emph{generalized dual map} $J_{p}:X\to 2^{X^*}$ is defined by
\begin{equation}\label{definition of genrealized dual map}
J_{p}(x)=\{x^*\in X^*:\langle x^*,x \rangle=\|x\|\cdot\|x^*\|,~\|x^*\|=\|x\|^{p-1}\},~x\in X.
\end{equation}
\end{defn}
The following proposition provides a connection between subdifferential of a convex function and the generalized dual map. 
\begin{prop}(\cite[p. 32]{Ch})
For $p>1$, the generalized dual map $J_{p}$ is the subdifferential of convex function $\frac{1}{p}\|\cdot\|^{p}$.
\end{prop}
We would like to mention here that for a $p$-uniformly smooth Banach space $X$ ($1<p\leq 2$) the generalized dual map $J_{p}$ is indeed a mapping defined from $X$ to $X^{*}$. In fact, if $X$ is a $p$-uniformly smooth space, then its dual space $X^{*}$ is of $q$-uniformly convex with $\frac{1}{p}+\frac{1}{q}=1$ (see \cite[p. 46]{Ch}). Then, for any $x^{*}_{1}$ and $x^{*}_{2}\in J_{p}(x)$, we have
\[
\langle x^{*}_{1},x \rangle=\langle x^{*}_{2},x \rangle=\|x\|^{p},
\]
with $\|x^{*}_{1}\|=\|x^{*}_{2}\|=\|x\|^{p-1}$. Therefore,
\[
\langle\frac{x^{*}_{1}+x^{*}_{2}}{2},x\rangle=\|x\|^{p},
\]
which implies that $\big{\|}\frac{x^{*}_{1}+x^{*}_{2}}{2}\big{\|}\geq \|x\|^{p-1}$. On the other hand, by the triangle inequality, we have $\big{\|}\frac{x^{*}_{1}+x^{*}_{2}}{2}\big{\|}\leq \|x\|^{p-1}$, which yields that $\big{\|}\frac{x^{*}_{1}+x^{*}_{2}}{2} \big{\|}=\|x\|^{p-1}$. Hence, it follows from the uniform convexity of $X^{*}$ that
$x^{*}_{1}=x^{*}_{2}$.

The following theorem provides a characterization of $p$-uniformly smooth spaces by using the generalized dual map $J_{p}$.

\begin{thm}(\cite[p. 48]{Ch})\label{equivalent smooth and dual mapping}
Let $1<p\leq 2$ and $X$ be a Banach space, then the following are equivalent to each other.
\begin{enumerate}[i)]
\item $X$ is $p$-uniformly smooth;
\item There is a constant $c>0$ satisfies that for all $x,~y\in X$,
\begin{equation}\label{constant}
\|x+y\|^{p}\leq \|x\|^{p}+p\langle J_{p}(x),y\rangle+c\|y\|^{p}.
\end{equation}
\end{enumerate}
\end{thm}
\begin{rem}
From the proof of Theorem \ref{equivalent smooth and dual mapping} we know that the constant $c>0$ in \eqref{constant} can be chosen to satisfy that $c=O(s_{p}(X))$.
\end{rem}

At the end of this subsection, we recall the following renorming theorem, due to Pisier \cite{Pis1}, which provides a characterization of $p$-uniformly smooth spaces by martingale inequalities.

\begin{thm}(\cite[Theorem 3.1]{Pis1})\label{Pisier's renorming theorem}
Fix $1<p\leq 2$. Then for a Banach space $X$ the following are equivalent to each other.
\begin{enumerate}[i)]
\item $X$ is linear isomorphic to a $p$-uniformly smooth space;
\item There is a constant $c>0$ such that all $X$-valued $L^{p}$-martingale $f=(f_{n})_{n=0}^{\infty}$ the following inequality holds
\begin{equation}
\sup\limits_{n\in\N}\E \|f_{n}\|^{p}\leq C^{p}\sum\limits_{n=0}^{\infty}\E\|d_{n,f}\|^{p};
\end{equation}
\item Same as ii) for all $X$-valued dyadic martingales.
\end{enumerate}
\end{thm}

\section{Azuma-type inequalities for $p$-uniformly smooth space-valued martingales}\label{main section}

In this section we will study Azuma-type inequalities for martingales with values in $p$-uniformly smooth space. In addition, as we will show below, the $p$-uniform smoothness of the image space is necessary for Azuma-type inequalities to hold. At the end of this section, Azuma-type inequality for self-normalized sums will also be discussed. In the rest of this paper, we will simply use $|\cdot|$ to stand for either Euclidean or Hilbert norm.

We now briefly recall the definition of functions of ``Mittag-Leffler type''. For $\alpha,~\beta>0$ and $z\in\C$, then
\begin{equation}\nonumber
E_{\alpha,\beta}(z)=\sum\limits_{k=0}^{\infty}\frac{z^{k}}{\Gamma(\alpha k+\beta)}
\end{equation}
is said to be two-parameter function of ``\emph{Mittag-Leffler type}'', which plays an important role in the fractional calculus (see \cite{Is}).

Recall that a random process $w=(w_{j})_{j=1}^{\infty}$ is called predictable relatively to the filtration $(\mathscr{F}_{j})_{j=0}^{\infty}$ if $w_{j}$ is $\mathscr{F}_{j-1}$-measurable for each $j\in\N$. We will simply call $w=(w_{j})_{j=1}^{\infty}$ predictable if there no confusion occurs. The following theorem is a simple variant of Pinelis' Azuma inequality (see \cite[Theorem 3.6]{Pin2}) for $2$-uniformly smooth space-valued conditionally symmetric martingales. The proof presented below is elementary and is slightly simpler than that in \cite{Pin2}. Hilbert space-valued martingales will suffice for our needs and in this case fruitful structures of Hilbert space could made the proof easier. In particular, we could compute the differential of $|\cdot|^{2}$ directly via the parallelogram law without the approximation argument (i.e. \cite[Lemma 2.2]{Pin2}).



\begin{thm}\label{Hilbert valued martingale concentration inequality}
Suppose that $f=(f_{j})_{j=0}^{\infty}$ is a $\HS$-valued martingale \big{(}relatively to the filtration $(\mathscr{F}_{j})_{j=0}^{\infty}$\big{)} such that there exists a nonnegative predictable sequence $w=(w_{j})_{j=1}^{\infty}$ which satisfies $|d_{j,f}|\leq w_{j}~a.s.$ for each $j\in \N$. Then, for $b\geq\Big{\|}\sum_{j=1}^{\infty} w^{2}_{j} \Big{\|}_{\infty}$ and every $r>0$, we have
\begin{equation}\nonumber
\Pb\{f^*\geq r\}\leq 2\exp\Big{\{}-\frac{r^{2}}{2b}\Big{\}}.
\end{equation}
\end{thm}
\begin{proof}
For a $C^{2}$-function $u$ defined on $\R$, let $\phi(t)=\cosh u(t)=E_{2,1}(u^{2}(t))$. Then,
\begin{equation}\label{twice derivate}
\begin{split}
\phi^{''}&=\big{(}\sum\limits_{k=1}^{\infty}\frac{u^{2k-1}\cdot u^{'}}{\Gamma(2k)}\big{)}^{'}\\
         &=\sum\limits_{k=1}^{\infty}\frac{(2k-1)u^{2k-2}\cdot(u^{'})^{2}}{\Gamma(2k)}+\sum\limits_{k=1}^{\infty}\frac{u^{2k-1}\cdot u^{''}}{\Gamma(2k)}\\
         &=\sum\limits_{k=0}^{\infty}\frac{u^{2k}\cdot (u^{'})^{2}}{\Gamma(2k+1)}+\sum\limits_{k=0}^{\infty}\frac{u^{2k}\cdot u\cdot u^{''}}{\Gamma(2k+2)}\\
         &=E_{2,1}(u^{2})(u^{'})^{2}+E_{2,2}(u^{2})\cdot u\cdot u^{''}\\
         &\leq E_{2,1}(u^{2})\cdot (u^{2})^{''}
\end{split}
\end{equation}
For a $\HS$-valued martingale $f=(f_{j})_{j=0}^{\infty}$ and $\lambda>0$, without loss of generality, suppose that $f_{0}=0~a.s.$, we define $\varphi(t)\coloneqq \E_{j-1}\cosh(\lambda|f_{j-1}+td_{j,f}|)$. Then, by the linearity of conditional expectation and \eqref{twice derivate} we have
\begin{equation}\label{estimate of twice derivate}
\varphi^{''}(t)\leq \lambda^{2}\E_{j-1}\Big{\{} E_{2,1}(|f_{j-1}+td_{j,f}|^{2})\cdot (|f_{j-1}+td_{j,f}|^{2})^{''}\langle d_{j,f},d_{j,f}\rangle\Big{\}}.
\end{equation}
By the parallelogram law it follows that $(|x+ty|^{2})^{''}\langle y,y\rangle=|y|^{2}$ for all $x,~y\in\HS$ and $t\in[0,1]$.
Therefore, \eqref{estimate of twice derivate} becomes into
\begin{equation}\nonumber
\varphi^{''}(t)\leq \lambda^{2}\E_{j-1}\Big{\{} E_{2,1}(|f_{j-1}+td_{j,f}|^{2})\cdot |d_{j,f}|^{2}\Big{\}},
\end{equation}
for all $t\in[0,1]$.

Note that $|d_{j,f}|\leq w_{j}$ and $w_{j}$ is $\mathscr{F}_{j-1}$-measurable for each $j\in\N$, then we have
\begin{equation}\label{estimate of twice derivate via w}
\varphi^{''}(t)\leq (\lambda\cdot w_{j})^{2}\cdot \varphi(t),
\end{equation}
for all $t\in[0,1]$. Since $\varphi^{'}(0)=0$, then solving this differential inequality (i.e. Gronwall inequality) yields that
\begin{equation}\label{super-martingale}
\varphi(1)\leq \exp\left\{\frac{\lambda^{2}w_{j}^{2}}{2} \right\}\varphi(0),
\end{equation}
that is, $\E_{j-1}\cosh(\lambda|f_{j}|)\leq \exp\Big{\{}\frac{\lambda^{2}w_{j}^{2}}{2} \Big{\}}\cosh(\lambda|f_{j-1}|)$ for all $j\in\N$. Let
\[
g_{n}=\frac{\cosh(\lambda|f_{n}|)}{\prod\limits_{j=1}^{n}\exp\{(\lambda\cdot w_{j})^{2}/2\}}
\]
for all $n\in\N_{0}$, then $(g_{n})_{n=0}^{\infty}$ forms a nonnegative super-martingale. Define $\sigma:(\Omega,\mathscr{F},\Pb)\to \N\cup\{\infty\}$ by $\sigma=\inf\{n\in\N_{0}:|f_{n}|\geq r\}$, then, it is clear that $\sigma$ is a stopping time.  Hence, by the ``Optional Sampling Theorem'' (Theorem \ref{optional sampling theorem}) and the Fatou's lemma, we have
\begin{equation}\label{super-martingale inequality 1}
\int_{\{\sigma<\infty\}}g_{\sigma}d\Pb\leq \E(g_{0})=1.
\end{equation}
Combining $b\geq\big{\|}\sum\limits_{j=1}^{\infty} w^{2}_{j}\big{\|}_{\infty}$ with \eqref{super-martingale inequality 1}, we have
\[
\Pb\{f^{*}\geq r\}\leq \frac{\exp\big{\{}b\lambda^{2}/2\big{\}}}{\cosh(\lambda r)}
\]
for all $r>0$ and $\lambda>0$.

Therefore, the fact $\cosh(\lambda r)\geq \frac{e^{\lambda r}}{2}$ yields that
\begin{equation}\label{estimate of Hilbert valued inequality}
\Pb\{f^{*}\geq r\}\leq\inf\limits_{\lambda>0}\Big{\{} 2\exp\big{\{} b\lambda^{2}/2-\lambda\cdot r \big{\}} \Big{\}}.
\end{equation}
Minimize \eqref{estimate of Hilbert valued inequality}, by choosing $\lambda=\frac{r}{b}>0$, we have
\begin{equation}\nonumber
\Pb\{f^{*}\geq r\}\leq 2\exp\Big{\{} -\frac{r^{2}}{2b} \Big{\}}.
\end{equation}
\end{proof}

The next lemma is a ``martingale dimension reduction'' argument, which reduces $p$-uniformly smooth space-valued martingale to an $\R^{2}$-valued one. The first ``martingale dimension reduction'' argument was provided by Kallenberg and Sztencel  in their study of continuous-time Hilbert space-valued martingales. They stated that any inequality for $\R^{2}$-valued continuous-time martingales which only involves some basic processes of the martingales remains true for arbitrary dimension with the same values of all constants (see \cite[Theorem 3.1]{K-S} for the details). However, the original proof of this theorem is long and quiet difficult, while Kwapie\'n and Woyczynski \cite[Proposition 5.8.3]{K-W} provided a simple proof of this theorem for discrete-time martingales. More precisely, given a Hilbert space-valued martingale $(f_{j})_{j=0}^{\infty}$, by Kawapie\'n and Woyczynski's construction, one could construct an $\R^{2}$-valued martingale $(N_{j})_{j=0}^{\infty}$ satisfies that
\[
|f_{j}|=|N_{j}|~~~{\rm and}~~~|d_{j,f}|=|d_{j,N}|
\]
for all $j\in \N$. The main idea of Kwapie\'n and Woyczynski's construction is to define an $\R^{2}$-valued martingale $(N_{j})_{j=0}^{\infty}$ inductively such that $\langle d_{j,f},f_{j-1}\rangle=\langle d_{j,N}, N_{j-1}\rangle$ and $|d_{j,f}|=|d_{j,N}|$ for all $j\in \N$. Then the structure of Euclidean norm, i.e.
\begin{equation}\label{orthogonality}
|x+y|^{2}=|x|^{2}+2\langle x,y\rangle+|y^{2}|~~\forall x,~y\in \R^{2},
\end{equation}
ensures that $(N_{j})_{j=0}^{\infty}$ satisfies
\[
|N_{j}|^{2}=|N_{j-1}|^{2}+2\langle N_{j-1},d_{j,N}\rangle+|d_{j,N}|^{2}~~\forall j\in \N.
\]
Then, by induction, one gets the $\R^{2}$-valued martingale $(N_{j})_{j=0}^{\infty}$ as required. However, the Kwapie\'n and Woyczynski's approach relies on orthogonality and \eqref{orthogonality}, which are only available in Hilbert space. In their remarkable work on Markov type and  threshold embeddings of metric spaces, Ding, Lee and Peres
generalized the Kwapie\'n and Woyczynski's method by providing a ``martingale dimension reduction'' argument for martingales with values in general $p$-uniformly smooth space. The key insight here, due to Ding, Lee and Peres, is that for martingales with values in $p$-uniformly smooth space, one can use the inequality \eqref{constant} instead of \eqref{orthogonality} (which is only available for Hilbert space) in the construction. And if one could control $N_{j}-N_{j-1}$ appropriately during their construction, then the constructed $\R^{2}$-valued martingale $(N_{j})_{j=0}^{\infty}$ will satisfy the desired properties (see \cite[Lemma 4.2]{D-L-P} for the details).

Our following lemma is inspired by the construction of Ding, Lee and Peres \cite[Lemma 4.2]{D-L-P} and is even simpler than that in \cite{D-L-P} with better constants. However, our approach is only available for certain types of martingales (such as dyadic martingales, martingales with bounded differences). Before turning into the technical details we shall describe the main idea of our construction.

Given a $p$-uniformly smooth space-valued martingale $(f_{j})_{j=0}^{\infty}$ with $f_{0}=0~a.s.$ we shall put $N_{0}=0~a.s.$ and construct $(N_{j})_{j=0}^{\infty}$ inductively as follows. Suppose that $(N_{j})_{j=0}^{n-1}$ have been constructed, we define $N_{j}$ as the sum of two orthogonal components. The first component is to ensure that $2\langle N_{j-1},N_{j}-N_{j-1}\rangle$ ``roughly'' equals to $p\langle J_{p}(f_{j-1}),f_{j}-f_{j-1}\rangle$ (see \eqref{preserves inner product} below). However, by some intrinsic differences between $p$-uniformly smooth spaces and Hilbert space, these two terms as above could not equal to each other in general. Hence, an appropriate truncation will be used to make sure that the constructed $N_{j}$ ``under controlled''. And in order to compensate the information lost by truncation, we shall put the lost information to the second component (the orthogonal part of the first component). Finally, to make the constructed sequence $(N_{j})_{j=0}^{\infty}$ still a martingale, we shall enlarge the filtration and attach a random sign in an appropriate way.

To organize our proof better, for $x=(x_{1},x_{2})\in \R^{2}$, we define $x^{\bot}=(x_{1}^{\bot},x_{2}^{\bot})$ as follows.
\begin{equation}\nonumber
(x^{\bot}_{1},x^{\bot}_{2})=
\begin{cases}
(-x_{2},x_{1})~~~~~{\rm if}~x\not=0,\\
(1,0),~~~~~~~~~~{\rm otherwise}.
\end{cases}
\end{equation}
It is clear that $\langle x,x^{\bot}\rangle=0$ for all $x\in \R^{2}$.
\begin{lem}\label{martingale dimension reduction lemma}
Let $f=(f_{j})_{j=0}^{\infty}$ be a martingale \big{(}relatively to the filtration $(\mathscr{F}_{j})_{j=0}^{\infty}$\big{)} with values in $p$-uniformly smooth space $X$ ($1<p\leq 2$) such that there exists a nonnegative predictable sequence $w=(w_{j})_{j=1}^{\infty}$ with $\|d_{j,f}\|\leq w_{j}~a.s.$ for all $j\in\N$. Then, there exists an $\R^{2}$-valued martingale $(N_{j})_{j=0}^{\infty}$ (relatively to another filtration) such that for all $j\in\N$,
\begin{enumerate}[i)]
\item $\|f_{j}-f_{0}\|^{p}\leq |N_{j}-N_{0}|^{2}$,
\item $|N_{j}-N_{j-1}|^{2}\leq K w_{j}^{p}$,
\end{enumerate}
where $K=O(s_{p}(X))$.
\end{lem}
\begin{proof}
Let $(\varepsilon_{j})_{j=1}^{\infty}$ be a sequence of independent Rademacher random variables and enlarge $\mathscr{F}_{j}$ by $\mathscr{F}^{\prime}_{j}=\sigma\big{(}\mathscr{F}_{j},\varepsilon_{1},\cdots,\varepsilon_{j} \big{)}$, i.e.
the $\sigma$-algebra generated by $\mathscr{F}_{j}$ and Rademacher random variables $\{\varepsilon_{1},\cdots,\varepsilon_{j}\}$ for all $j\in \N$. Then, it is clear that $(f_{j})_{j=0}^{\infty}$ is still a martingale and $(w_{j})_{j=1}^{\infty}$ is still predictable relatively to the new filtration $(\mathscr{F}^{\prime}_{j})_{j=0}^{\infty}$.


We now construct the martingale $(N_{j})_{j=0}^{\infty}$ inductively as follows. Without loss of generality we assume that $f_{0}=0~a.s.$ and put $N_{0}\coloneqq 0~a.s.$. Suppose that $(N_{n})_{n=0}^{j-1}$ have been constructed, then, let $A_{j-1}=\{w_{j}^{p}\leq |N_{j-1}|^{2}\}$ (a suitable truncation, which in order to control the constructed martingale appropriately, see \eqref{preserves inner product}, \eqref{estimate N 2} and \eqref{upper bound for differences} below) and define

\begin{equation}\label{definition of martingale}
\begin{split}
N_{j}-N_{j-1}&=\frac{p\langle J_{p}(f_{j-1}),f_{j}-f_{j-1}\rangle \cdot\mathds{1}_{A_{j-1}}}{2}\cdot\frac{N_{j-1}}{|N_{j-1}|^{2}+\mathds{1}_{\{N_{j-1}=0\}}}\\
             &+\sqrt{c+p}\cdot\varepsilon_{j}w_{j}^{\frac{p}{2}}\cdot\frac{N_{j-1}^{\bot}}{|N_{j-1}|+\mathds{1}_{\{N_{j-1}=0\}}}.
\end{split}
\end{equation}

Since $\langle N_{j-1},N_{j-1}^{\bot}\rangle=0$ and $A_{j-1}=\{w_{j}^{p}\leq |N_{j-1}|^{2}\}$ , then, as we expect,
\begin{equation}\label{preserves inner product}
2\langle N_{j-1},N_{j}-N_{j-1}\rangle=p\langle J_{p}(f_{j-1}),f_{j}-f_{j-1}\rangle\cdot\mathds{1}_{A_{j-1}}.
\end{equation}

Firstly, we verify that $(N_{j})_{j=0}^{\infty}$ is a martingale relatively to the filtration $(\mathscr{F}_{j}^{\prime})_{j=0}^{\infty}$.
Indeed, note here that $w_{j}$ is $\mathscr{F}^{\prime}_{j-1}$-measurable and $\varepsilon_{j}$ is independent of $\mathscr{F}_{j-1}^{\prime}$. Then, by the linearity of conditional expectation, we have
\begin{equation}\nonumber
\begin{split}
\E(N_{j}-N_{j-1}|\mathscr{F}^{\prime}_{j-1})&=\frac{p}{2}N_{j-1}\cdot\E\Big{(}\frac{\langle J_{p}(f_{j-1}),f_{j}-f_{j-1}\rangle \cdot\mathds{1}_{A_{j-1}}}{|N_{j-1}|^{2}+\mathds{1}_{\{N_{j-1}=0\}}}|\mathscr{F}^{\prime}_{j-1}\Big{)}\\
                                            &+\sqrt{c+p}\cdot w_{j}^{\frac{p}{2}}\E(\varepsilon_{j})\cdot \frac{N_{j-1}^{\bot}}{|N_{j-1}|+\mathds{1}_{\{N_{j-1}=0\}}}\\
                           &=\frac{p}{2}N_{j-1}\cdot\frac{\langle J_{p}(f_{j-1}),\E(f_{j}-f_{j-1}|\mathscr{F}^{\prime}_{j-1})\rangle\cdot \mathds{1}_{A_{j-1}}}{|N_{j-1}|^{2}+\mathds{1}_{\{N_{j-1}=0\}}}\\
                           &=0,
\end{split}
\end{equation}
for all $j\in\N$, which yields that $(N_{j})_{j=0}^{\infty}$ is a martingale relatively to the filtration $(\mathscr{F}^{\prime}_{j})_{j=0}^{\infty}$.

Secondly, we shall verify that the $\R^{2}$-valued martingale $(N_{j})_{j=0}^{\infty}$ satisfies the desired properties $i)$ and $ii)$. In order to do this the proof will be divided into two cases.

{\bf Case {\uppercase\expandafter{\romannumeral1}}.} On the event $\{\omega\in \Omega:N_{j-1}(\omega)\not=0\}$.

Since $X$ is a $p$-uniformly smooth space, then, Theorem \ref{equivalent smooth and dual mapping} yields that for all $j\in \N$
\begin{equation}\label{p-uniform smoothness}
\begin{split}
\|f_{j}\|^{p}&\leq \|f_{j-1}\|^{p}+p\langle J_{p}(f_{j-1}),f_{j}-f_{j-1}\rangle+c\|d_{j,f}\|^{p}\\
             &\leq \|f_{j-1}\|^{p}+p\langle J_{p}(f_{j-1}),f_{j}-f_{j-1}\rangle+cw_{j}^{p},
\end{split}
\end{equation}
where $c=O(s_{p}(X))$.

Suppose that $i)$ holds true for all $k=0,\cdots,j-1$, then \eqref{p-uniform smoothness} implies that
\begin{equation}\label{right hand side}
\|f_{j}\|^{p}\leq |N_{j-1}|^{2}+p\langle J_{p}(f_{j-1}),f_{j}-f_{j-1}\rangle+cw_{j}^{p},
\end{equation}
with $c=O(s_{p}(X))$.

On the other hand, since $\langle N_{j-1},N_{j-1}^{\bot}\rangle=0$, we have
\begin{equation}\label{estimate N 1}
\begin{split}
|N_{j}|^{2}&=|N_{j-1}|^{2}\Big{(} 1+\frac{p}{2}\cdot\frac{\langle J_{p}(f_{j-1}),f_{j}-f_{j-1}\rangle \cdot\mathds{1}_{A_{j-1}}}{|N_{j-1}|^{2}} \Big{)}^{2}+(c+p)w_{j}^{p}\\
                 &\geq |N_{j-1}|^{2}\Big{(}1+p\cdot\frac{\langle J_{p}(f_{j-1}),f_{j}-f_{j-1}\rangle \cdot\mathds{1}_{A_{j-1}}}{|N_{j-1}|^{2}} \Big{)}+(c+p)w_{j}^{p}\\
                 &=|N_{j-1}|^{2}+p\cdot\langle J_{p}(f_{j-1}),f_{j}-f_{j-1}\rangle \cdot\mathds{1}_{A_{j-1}}+(c+p)w_{j}^{p}.
\end{split}
\end{equation}
Then, it now suffices to estimate the term $p\langle J_{p}(f_{j-1}),f_{j}-f_{j-1}\rangle \cdot\mathds{1}_{A_{j-1}^{c}}$, where $A_{j-1}^{c}$ is the complementary set of $A_{j-1}$. Note that $A_{j-1}^{c}=\{w_{j}^{p}>|N_{j-1}|^{2}\}$, we have
\begin{equation}\label{estimate N 2}
\begin{split}
p\langle J_{p}(f_{j-1}),f_{j}-f_{j-1}\rangle \cdot\mathds{1}_{A_{j-1}^{c}}&\leq p\|f_{j-1}\|^{p-1}\|f_{j}-f_{j-1}\|\cdot\mathds{1}_{A_{j-1}^{c}}\\
                                                                         &\leq p|N_{j-1}|^{\frac{2(p-1)}{p}}\|f_{j}-f_{j-1}\|\cdot\mathds{1}_{A_{j-1}^{c}}\\
                                                                         &\leq pw_{j}^{p},
\end{split}
\end{equation}
Then, combining with \eqref{right hand side}, \eqref{estimate N 1} and \eqref{estimate N 2} it follows that $\|f_{j}\|^{p}\leq |N_{j}|^{2}$, which completes our proof of $i)$.

To prove that $(N_{j})_{j=0}^{\infty}$ satisfies $ii)$, we note here that $\langle N_{j-1},N_{j-1}^{\bot}\rangle=0$, then,
\begin{equation}\nonumber
|N_{j}-N_{j-1}|^{2}=\frac{p^{2}}{4}\cdot\frac{\langle J_{p}(f_{j-1}),f_{j}-f_{j-1}\rangle^{2}\cdot \mathds{1}_{A_{j-1}}}{|N_{j-1}|^{2}}+(c+p)w_{j}^{p}.
\end{equation}
By $i)$, it follows that
\begin{equation}\label{upper bound for differences}
\begin{split}
|N_{j}-N_{j-1}|^{2}&\leq \frac{p^{2}}{4}\cdot\frac{\|f_{j-1}\|^{2(p-1)}\|f_{j}-f_{j-1}\|^{2}\cdot\mathds{1}_{A_{j-1}}}{|N_{j-1}|^{2}}+(c+p)w_{j}^{p}\\
                         &\leq \frac{p^{2}}{4}\cdot\frac{|N_{j-1}|^{\frac{4(p-1)}{p}}w_{j}^{2}\cdot\mathds{1}_{A_{j-1}}}{|N_{j-1}|^{2}}+(c+p)w_{j}^{p}\\
                         &=\frac{p^{2}}{4}\cdot\frac{w_{j}^{2}\cdot\mathds{1}_{A_{j-1}}}{|N_{j-1}|^{\frac{2(2-p)}{p}}}+(c+p)w_{j}^{p},
\end{split}
\end{equation}
Note that $1<p\leq 2$ and $A_{j-1}=\{w_{j}^{p}\leq |N_{j-1}|^{2}\}$, then $\frac{4-2p}{p}\geq 0$ and \eqref{upper bound for differences} can be further estimated as follows
\begin{equation}\nonumber
\begin{split}
|N_{j}-N_{j-1}|^{2}&\leq \frac{p^{2}}{4}\cdot\frac{w_{j}^{2}}{w_{j}^{2-p}}+(c+p)w_{j}^{p}\\
                         &=\big{(}\frac{p^{2}}{4}+c+p \big{)}w_{j}^{p}
\end{split}
\end{equation}
Denote that $K=\frac{p^{2}}{4}+c+p\leq c+3$, then
\[
|N_{j}-N_{j-1}|^{2}\leq Kw_{j}^{p},~\forall j\in \N
\]
with $K=O(s_{p}(X))$.

{\bf Case {\uppercase\expandafter{\romannumeral2}}.} On the event $\{\omega\in \Omega:N_{j-1}(\omega)=0\}$, it is easy to verify that the martingale $(N_{j})_{j=0}^{\infty}$ satisfies the desired properties $i)$ and $ii)$.

Indeed, by induction, we have $\|f_{j-1}-f_{0}\|^{p}\leq |N_{j-1}-N_{0}|^{2}$, which yields that $f_{j-1}=0$ on the event $\{\omega\in\Omega:N_{j-1}(\omega)=0\}$. Hence,
\begin{equation}\nonumber
|N_{j}-N_{j-1}|^{2}=|N_{j}|^{2}=(c+p)w_{j}^{p}\geq (c+p)\|f_{j}-f_{j-1}\|^{p}\geq \|f_{j}\|^{p},
\end{equation}
which completes our proof.
\end{proof}

With the help of the Lemma \ref{martingale dimension reduction lemma} we can now prove the following Azuma-type inequality for $p$-uniformly smooth space-valued martingales.

\begin{thm}\label{predictable sequence dominated case}
Let $f=(f_{j})_{j=0}^{\infty}$ be a martingale with values in $p$-uniformly smooth space ($1<p\leq 2$) such that there exists a non-negative predictable sequence $w=(w_{j})_{j=1}^{\infty}$ with $\|d_{j,f}\|\leq w_{j}$ for all $j\in\N$. Then, there exists a constant $K$ depending only on $X$ such that for all $r\geq 0$,
\begin{equation}\nonumber
\Pb\{f^{*}\geq r\}\leq 2 \exp\left\{-\frac{r^{p}}{2K\Big{\|}\sum_{j=1}^{\infty}w_{j}^{p}\Big{\|}_{\infty}} \right\},
\end{equation}
where $K=O(s_{p}(X))$.  
\end{thm}
\begin{proof}
Since $w=(w_{j})_{j=0}^{\infty}$ is a predictable sequence, then apply Lemma \ref{martingale dimension reduction lemma}, there exists an $\R^{2}$-valued martingale $N=(N_{j})_{j=0}^{\infty}$ (relative another filtration) satisfying
\begin{enumerate}[i)]
\item $\|f_{j}\|^{p}\leq |N_{j}|^{2}$,
\item $|N_{j}-N_{j-1}|^{2}\leq K w_{j}^{p}$,
\end{enumerate}
for all $j\in \N$, where $K=O(s_{p}(X))$.

By Theorem \ref{Hilbert valued martingale concentration inequality}, we have for all $r\geq 0$,
\begin{equation}\label{Hilbert space valued setting}
\Pb\{N^{*}\geq r\}\leq 2\exp\left\{-\frac{r^{2}}{2K\Big{\|}\sum_{j=1}^{\infty}w_{j}^{p}\Big{\|}_{\infty}} \right\}.
\end{equation}
By $i)$ we have $\{\|f_{n}\|\geq t^{\frac{2}{p}}\}\subseteq\{|N_{n}|^{2}\geq t^{2}\}=\{|N_{n}|\geq t\}$ for all $n\in\N$ and $t>0$. Therefore, for $r>0$, we have  $\{\|f_{n}\|\geq r\}\subseteq\{\|N_{n}\|\geq r^{\frac{p}{2}}\}$ and hence, combining with \eqref{Hilbert space valued setting}, it follows that
\begin{equation}
\Pb\{f^{*}\geq r\}\leq 2\exp\left\{-\frac{r^{p}}{2K\Big{\|}\sum_{j=1}^{\infty}w_{j}^{p}\Big{\|}_{\infty}} \right\}.
\end{equation}
This completes our proof.
\end{proof}
\begin{rem}
For the case $\Big{\|}\sum_{j=1}^{\infty}w_{j}^{p} \Big{\|}_{\infty}=\infty$ the inequality \ref{Azuma type} holds true trivially.
\end{rem}
The following two inequalities are easy corollaries of Theorem \ref{predictable sequence dominated case}. The first one is a generalization of Pinelis' inequality for $2$-uniformly smooth space-valued martingales \cite{Pin1} and the second one can be viewed as an improvement of Naor's Azuma inequality.
\begin{cor}\label{refinement of Azuma inequality}
Suppose that $f=(f_{j})_{j=0}^{\infty}$ is a conditionally symmetric martingale with values in a $p$-uniformly smooth space $X$ ($1<p\leq 2$). If $b$ is such that $b\geq\|S_{p}^{p}(f)\|_{\infty}$, then the following inequality holds for all $r\geq 0$,
\begin{equation}\nonumber
\Pb\{f^{*}\geq r\}\leq 2\exp\Big{\{}-\frac{r^{p}}{2Kb}\Big{\}},
\end{equation}
where $K$ depends only on $X$.
\end{cor}
\begin{proof}
Since $f=(f_{j})_{j=0}^{\infty}$ is a conditionally symmetric martingale with valued in a $p$-uniformly smooth space $X$, then, by Proposition \ref{Burkholder's replace filtration argument}, there exists a filtration $(\mathscr{G}_{j})_{j=0}^{\infty}$ such that $f=(f_{j})_{j=0}^{\infty}$ is a martingale and $(\|d_{j,f}\|)_{j=1}^{\infty}$ is predictable relatively to the filtration $(\mathscr{G}_{j})_{j=0}^{\infty}$.

Then, by Theorem \ref{predictable sequence dominated case}, we have for all $r\geq 0$,
\begin{equation}\nonumber
\Pb\{f^{*}\geq r\}\leq 2\exp\Big{\{}-\frac{r^{p}}{2Kb} \Big{\}},
\end{equation}
where $K=O(s_{p}(X))$. This completes our proof.
\end{proof}


\begin{cor}
Let $f=(f_{j})_{j=0}^{\infty}$ be martingale with values in $p$-uniformly smooth Banach space $X$ ($1<p\leq 2$). If $b$ is such that $b\geq\sum_{j=1}^{\infty}\|d_{j,f}\|_{\infty}^{p}$, then the following inequality holds for all $r\geq 0$,
\begin{equation}\nonumber
\Pb\{f^{*}\geq r\}\leq 2\exp\Big{\{}-\frac{r^{p}}{2Kb} \Big{\}},
\end{equation}
where $K$ depends only on $X$.
\end{cor}
\begin{proof}
Obviously $\big{(}\|d_{j,f}\|_{\infty} \big{)}_{j=0}^{\infty}$ is a predictable sequence such that $\|d_{j,f}\|\leq\|d_{j,f}\|_{\infty}$ for each $j\in\N$.
\end{proof}

Theorem \ref{predictable sequence dominated case} states that certain type of Azuma inequality holds for martingales with values in $p$-uniformly smooth space and it is natural to consider the inverse question: Is the $p$-uniform smoothness of the image space necessary for Azuma-type inequalities to hold? We shall answer this question affirmatively in the following. Before doing this we now introduce the definition of ``Azuma type'' for Banach spaces.
\begin{defn}
A Banach space $X$ is said to have \emph{Azuma type} $p$, if there exists a constant $K>0$ depending on $X$ such that for every martingale $f=(f_{j})_{j=0}^{\infty}$ with values in $X$ and every predictable sequence $w=(w_{j})_{j=1}^{\infty}$ with $\|d_{j,f}\|\leq w_{j}$ for all $j\in\N$, then the following inequality holds, for $r>0$,
\begin{equation}\label{Azuma type}
\Pb\{f^{*}\geq r\}\leq 2\exp\left\{-\frac{r^{p}}{K \Big{\|}\sum_{j=1}^{\infty}w_{j}^{p}\Big{\|}_{\infty}} \right\}.
\end{equation}
The smallest constant $K$ such that inequality \eqref{Azuma type} holds is called the Azuma type constant of $X$ denoted by $K_{p,X}$.
\end{defn}
\begin{rem}
The constant ``$2$'' that appears in the definition of ``\emph{Azuma type}'' does not have any special meaning, which can be replaced by any other absolute positive constants.
\end{rem}

In other word, the Theorem \ref{predictable sequence dominated case} asserts that $p$-uniformly smooth spaces are of Azuma type $p$. We now at the position to prove that a Banach space $X$ is of Azuma type $p$ then it must be $p$-uniformly smooth, up to a linear isomorphism. Our proof of this statement is based on a ``\emph{good-$\lambda$ inequality}'' of Burkholder \cite{Bu1}.

\begin{lem}\label{good lambda inequality}
Let $X$ be a Banach space with Azuma type $p$ with constant $K_{p,X}$, then for every conditionally symmetric martingale $f=(f_{j})_{j=0}^{\infty}$ \big{(}relatively to the filtration $(\mathscr{F}_{j})_{j=0}^{\infty}$\big{)} the following inequality holds. For $\beta>0$ and $0<\delta<\beta-1$ we have
\begin{equation}\nonumber
\Pb\{f^*>\beta\lambda, S_{p}(f)\leq \delta \lambda\}\leq 2 \exp\left\{-\frac{(\beta-1-\delta)^{p}}{K_{p,X} \delta^{p}} \right\}\Pb\{f^*>\lambda\},
\end{equation}
for all $\lambda>0$.
\end{lem}
\begin{proof}
For conditionally symmetric martingale $f=(f_{j})_{j=0}^{\infty}$, applying Proposition \ref{Burkholder's replace filtration argument}, there exits a filtration $(\mathscr{G}_{j})_{j=0}^{\infty}$ such that $(f_{j})_{j=0}^{\infty}$ is also a martingale relatively to the filtration $(\mathscr{G}_{j})_{j=0}^{\infty}$ and $\|d_{j,f}\|$ is $\mathscr{G}_{j-1}$-measurable for each $j\in\N$. Define stopping times as follows
\[
\mu=\inf\{n\in\N_{0}:\|f_{n}\|>\lambda\},
\]
\[
\nu=\inf\{n\in\N_{0}:\|f_{n}\|>\beta\lambda\},
\]
\[
\sigma=\inf\{n\in\N_{0}:S_{p,n+1}(f)>\delta\lambda\}.
\]
Then it is clear that $\{f^{*}>\lambda\}=\{\mu<\infty\}$, $\{f^{*}>\beta\lambda\}=\{\nu<\infty\}$ and $\{S_{p}(f)\leq \delta\lambda\}=\{\sigma=\infty\}$.
Let
\begin{equation}\nonumber
h_{n}=\sum\limits_{j=0}^{n}\mathds{1}_{\{\mu<j\leq \nu\wedge\sigma\}}d_{j,f},
\end{equation}
for all $n\in \N$ and $h_{0}\coloneqq 0~a.s.$ for convenience. We now verify that $h=(h_{j})_{j=0}^{\infty}$ is a martingale relatively to the filtration $(\mathscr{G}_{j})_{j=0}^{\infty}$. Indeed, it suffices to show that $\{\mu<j\leq\nu\wedge\sigma\}$ is $\mathscr{G}_{j-1}$-measurable for each $j\in\N$. Since
\begin{equation}\nonumber
\{\mu<n\}=\bigcup\limits_{j=0}^{n-1}\{\|f_{j}\|>\lambda\},
\end{equation}
\begin{equation}\nonumber
\{n\leq \nu\}=\bigcap\limits_{j=0}^{n-1}\{\|f_{j}\|\leq \beta\lambda\},
\end{equation}
and
\begin{equation}\nonumber
\{n\leq\sigma\}=\bigcap\limits_{j=0}^{n}\{S_{p,j}(f)\leq \delta\lambda\},
\end{equation}
then by the fact that $\|d_{j,f}\|$ are $\mathscr{G}_{j-1}$-measurable we have $\{\mu<j\leq\nu\wedge\sigma\}$ is $\mathscr{G}_{j-1}$ measurable for each $j\in\N$, which entails that $h=(h_{j})_{j=0}^{\infty}$ is a martingale with $\|d_{j,h}\|$ is $\mathscr{G}_{j-1}$-measurable for each $j\in\N$.

By the definition of stopping times $\mu$, $\nu$ and $\sigma$, we have
\begin{equation}\nonumber
\{f^{*}>\beta\lambda,S_{p}(f)\leq \delta\lambda\}=\{\nu<\infty,\sigma=\infty\}\subseteq\{h^{*}>(\beta-1-\delta)\lambda\}.
\end{equation}

We now estimate the term $\|S_{p}(h)\|_{\infty}$. It is clear that $h=0~a.s.$ on $\{\mu=\infty\}$, then, by the definition of $h=(h_{j})_{j=0}^{\infty}$ and $\sigma$ we have
$S_{p}(h)\leq S_{p,\sigma}(f)\leq\delta\lambda$ on $\{\mu<\infty\}$ almost surely.

Hence, by the fact that $X$ is of Azuma type $p$ with Azum type constant $K_{p,X}$, then the following inequality holds
\begin{equation}\nonumber
\begin{split}
\Pb\{f^{*}>\beta\lambda,S_{p}(f)\leq \delta\lambda\}&\leq\Pb\{h^{*}>(\beta-1-\delta)\lambda\}\\
                                                    &=\int_{\{\mu<\infty\}}\mathds{1}_{\{h^{*}>(\beta-1-\delta)\lambda\}}d\Pb\\
                                                    &=\int\E(\mathds{1}_{\{h^{*}>(\beta-1-\delta)\lambda\}}\cdot\mathds{1}_{\{\mu<\infty\}}|\mathscr{G}_{\mu})d\Pb\\
                                                    &=\Pb\{h^{*}>(\beta-1-\delta)\lambda\}\cdot\Pb\{\mu<\infty\}\\
                                                    &\leq 2\exp\left\{-\frac{(\beta-1-\delta)^{p}}{K_{p,X}\delta^{p}}
                                                    \right\}\Pb\{f^{*}>\lambda\}.
\end{split}
\end{equation}
\end{proof}

Combining with the ``good-$\lambda$ inequality'' above and the renorming theorem of Pisier \cite{Pis1} (corresponding to Theorem \ref{Pisier's renorming theorem}), we can now assert that a Banach space which is of Azuma type $p$ must be linear isomorphic to a $p$-uniformly smooth space.
\begin{thm}\label{p-uniform smoothness is necessary}
Suppose that $X$ is a Banach space with Azuma type $p$ ($1<p\leq 2$). Then, $X$ is linear isomorphic to a $p$-uniformly smooth space.
\end{thm}
\begin{proof}
Suppose that $X$ has the Azuma type $p$ with Azuma type constant $K_{p,X}$. Then, by Lemma \ref{good lambda inequality}, for every $L^{p}$ dyadic martingale $f=(f_{j})_{j=0}^{\infty}$ with values in $X$, we have, for all $\lambda>0$,
\[
\Pb\{f^*>\beta\lambda, S_{p}(f)\leq \delta \lambda\}\leq 2\exp\left\{-\frac{(\beta-1-\delta)^{p}}{K_{p,X} \delta^{p}} \right\}\Pb\{f^*>\lambda\}.
\]

Choose $\beta=2$ and $\delta=\frac{1}{1+\big{(}(r+2)\log2 K_{p,X}\big{)}^{1/p}}$,
then, \cite[Lemma 7.1]{Bu1} yields that
\begin{equation}\nonumber
\|f^{*}\|_{r}^{r}\leq 4^{r+1}\big{(}(r+2)\log2\cdot K_{p,X}\big{)}^{\frac{r}{p}} \|S_{p}(f)\|_{r}^{r},~1\leq r<\infty.
\end{equation}
By the Pisier's renorming theorem (see Theorem \ref{Pisier's renorming theorem}), it follows that $X$ is linear isomorphic to
a $p$-uniformly smooth space.
\end{proof}
\begin{rem}
Combining with Theorem \ref{predictable sequence dominated case} and Theorem \ref{p-uniform smoothness is necessary}, we obtain a characterization of $p$-uniformly smooth spaces in terms of Azuma-type inequalities.
\end{rem}

We conclude this section by providing a further refinement of the Azuma-type inequality for Banach space-valued martingales. Observe that if $f=(f_{j})_{j=0}^{\infty}$ is a conditionally symmetric martingales with values in $p$-uniformly smooth space, then the Corollary \ref{refinement of Azuma inequality} yields that, for all $r\geq 0$,
\begin{equation}\label{Azuma inequality in normal sense}
\Pb\Big{\{}\|f_{n}-f_{0}\|\geq r \Big{\}}\leq 2\exp\left\{-\frac{r^{p}}{2K\big{\|}S_{p,n}(f)\big{\|}_{\infty}^{p}} \right\},
\end{equation}
where $K=O(s_{p}(X))$. Equivalently, \eqref{Azuma inequality in normal sense} can be restated as follows
\begin{equation}\nonumber
\Pb\left\{\frac{\|f_{n}-f_{0}\|}{\big{\|}S_{p,n}(f)\big{\|}_{\infty}}\geq r \right\}\leq 2\exp\Big{\{}-\frac{r^{p}}{2K}\Big{\}}.
\end{equation}

Note that $\left\{\frac{\|f_{n}-f_{0}\|}{\big{\|}S_{p,n}(f) \big{\|}_{\infty}}\geq r \right\}\subseteq \left\{\frac{\|f_{n}-f_{0}\|}{S_{p,n}(f)}\geq r \right\}$ for all $r\geq 0$, hence, the following Azuma-type inequality for self-normalized sums is an improvement of the classical Azuma inequality for conditionally symmetric martingales.

\begin{thm}\label{refinement of Azuma type inequality for self-normaized sums}
Let $f=(f_{n})_{n=0}^{\infty}$ be a conditionally symmetric martingale with values in $p$-uniformly smooth Banach space. Then, for every $r\geq 0$, the following inequality holds
\begin{equation}\label{self-normalized Auama type inequality}
\Pb\left\{\frac{\|f_{n}-f_{0}\|}{S_{p,n}(f)}\geq r \right\}\leq 4\exp\Big{\{}-\frac{r^{p}}{2K} \Big{\}},
\end{equation}
for all $n\in \N$, where $K=O(s_{p}(X))$.
\end{thm}
\begin{proof}
We begin our proof with a special case, namely Hilbert space-valued martingales such that the norm of martingale differences dominated by non-negative predictable sequence. And then we conclude our proof by applying the ``martingale dimension reduction'' lemma to reduce the case from $p$-uniformly smooth space-valued to the case of Hilbert space-valued. Without loss of generality we assume martingale $(f_{j})_{j=0}^{\infty}$ satisfies $f_{0}=0~a.s.$.

{\bf Step \uppercase\expandafter{\romannumeral1}.} Suppose that $f=(f_{j})_{j=0}^{\infty}$ be a Hilbert space-valued martingale such that there exists a predictable sequence $w=(w_{j})_{j=1}^{\infty}$ with $|d_{j,f}|\leq w_{j}$ for all $j\in\N$ and denote that $S_{2,n}(w)=\big{(}\sum_{j=1}^{n} w_{j}^{2} \big{)}^{1/2}$ for all $n\in \N$.

Let $B=\{\frac{|f_{n}|}{S_{2,n}(w)}\geq r\}$ for $r>0$, then by the fact that $e^{t}\leq 2\cosh(t)\leq 2e^{t}$ for all $t\in \R$, and the Cauchy-Schwarz inequality, we have the following

\begin{small}
\begin{equation}\label{estimate for self-normalized}
\begin{split}
\Pb(B)&\leq 2\inf\limits_{\lambda>0}\int\frac{\cosh\frac{\lambda}{2}|f_{n}|}{\exp\big{\{}\frac{\lambda r}{2}S_{2,n}(w)\big{\}}}\cdot\mathds{1}_{B}~d\Pb\\
      &= 2\inf\limits_{\lambda>0}\int \frac{\cosh\frac{\lambda}{2}|f_{n}|}{\exp\big{\{}\frac{\lambda^{2}}{4}S_{2,n}^{2}(w)\big{\}}}\cdot \exp\Big{\{}\frac{\lambda^{2}}{4} S_{2,n}^{2}(w)-\frac{\lambda r}{2}S_{2,n}(w)\Big{\}}\cdot\mathds{1}_{B}~d\Pb\\
      &\leq 2\inf\limits_{\lambda>0}\!\Big{(}\!\int\frac{\cosh \lambda |f_{n}|}{\exp\big{\{}\frac{\lambda^{2}}{2}S_{2,n}^{2}(w)\! \big{\}}}~d\Pb \Big{)}^{1/2}\!\cdot\!\Big{(}\!\int\exp\big{\{}\!\frac{\lambda^{2}}{2}S_{2,n}^{2}(w)\!-\!\lambda rS_{2,n}(w)\!\big{\}}\!\cdot\! \mathds{1}_{B}~ d\Pb\Big{)}^{1/2}.
\end{split}
\end{equation}
\end{small}

By the proof of Theorem \ref{Hilbert valued martingale concentration inequality}, we know that $\left\{\frac{\cosh \lambda |f_{n}|}{\exp\big{\{}\frac{\lambda^{2}}{2} S_{2,n}^{2}(w)\big{\}}} \right\}_{n=0}^{\infty}$ forms a super-martingale. Then, \eqref{estimate for self-normalized} becomes into
\begin{equation}\nonumber
\begin{split}
\Pb(B)&\leq 2\inf\limits_{\lambda>0}\Big{(}\int \exp\big{\{}\frac{\lambda^{2}}{2}S_{2,n}^{2}(w)-\lambda rS_{2,n}(w)\big{\}}\cdot \mathds{1}_{B}~d\Pb \Big{)}^{1/2}\\
      &=2\int \big{(}\exp\big{\{}-\frac{r^{2}}{2}\big{\}}\cdot \mathds{1}_{B}~d\Pb\big{)}^{1/2}\\
      &=2\exp\big{\{}-\frac{r^{2}}{4} \big{\}}\cdot\Pb(B)^{1/2}.
\end{split}
\end{equation}
Therefore,
\begin{equation}\label{self-normalized sums}
\Pb\left\{\frac{|f_{n}|}{S_{2,n}(w)}\geq r \right\}=\Pb(B)\leq 4\exp\big{\{}-\frac{r^{2}}{2} \big{\}},
\end{equation}
for all $n\in \N$.

{\bf Step \uppercase\expandafter{\romannumeral2}.} Suppose that $f=(f_{j})_{j=0}^{\infty}$ is a conditionally symmetric martingale with values in $p$-uniformly smooth space, then there exists another filtration $(\mathscr{G}_{j})_{j=0}^{\infty}$ such that $f=(f_{j})_{j=0}^{\infty}$ is a martingale relatively to the filtration $\mathscr{G}$ and $\|d_{j,f}\|$ is $\mathscr{G}_{j-1}$-measurable for each $j\in\N$. By Lemma \ref{martingale dimension reduction lemma}, there exists an $\R^{2}$-valued martingale $\{N_{j}\}_{j=0}^{\infty}$ such that for all $j\in \N$
\begin{enumerate}[i)]
\item $\|f_{j}\|^{p}\leq |N_{j}|^{2}$,
\item $|N_{j}-N_{j-1}|^{2}\leq K \|f_{j}-f_{j-1}\|^{p}$,
\end{enumerate}
where $K=O(s_{p}(X))$.

Then, applying \eqref{self-normalized sums} we have the following
\begin{equation}\nonumber
\begin{split}
\Pb\left\{\frac{\|f_{n}\|}{S_{p,n}(f)}\geq r \right\}&=\Pb\left\{\frac{\|f_{n}\|^{p}}{S_{p,n}^{p}(f)}\geq r^{p} \right\}\\
                                                      &\leq\Pb\left\{\frac{|N_{n}|^{2}}{S_{2,n}^{2}(N_{n})}\geq\frac{r^{p}}{K}\right\}\\
                                                      &\leq 4 \exp\big{\{}-\frac{r^{p}}{2K} \big{\}},
\end{split}
\end{equation}
for all $n\in \N$, which completes our proof.
\end{proof}

\begin{rem}
The constant ``$4$'' of \eqref{self-normalized Auama type inequality} is not the best possible and it will be ``$2$'' for real-valued conditionally symmetric martingales. However, if we ignore the universal constant here, the inequality \eqref{self-normalized Auama type inequality} indeed provides the right degree of concentration.
\end{rem}

\section{Some further inequalities for Banach space-valued martingales}

In this section, methods developed in the previous section will be used to derive some further inequalities for Banach space-valued martingales, such as moment inequalities for double dyadic martingales and De la Pe\~{n}a-type inequality for conditionally symmetric martingales.


\subsection{Moment inequalities for Banach space-valued double dyadic martingales}

Moment inequalities for double dyadic martingales are important to Harmonic analysis. Pipher \cite{Pip} provided a moment inequality for double dyadic martingales (see the definition below) and then applied it to the study the exponential square integrability of $|f-f_{Q}|$ over $Q$ in the bidisc case of two parameter kernel. And Ba\~{n}uelos \cite{Ba} extended Phipher's results to continuous-time martingales on Brownian filtration and then applied them to the study of Riesz transforms.


In the rest of this subsection we will apply the methods developed in Section \ref{main section} to deduce moment inequalities for Banach space-valued double dyadic martingales. Recall that a random process $f=(f_{n})_{n=0}^{\infty}$ is said to be \emph{dyadic martingale} (or, \emph{Paley-Walsh martingale}) if $f=(f_{n})_{n=0}^{\infty}$ is martingale relatively to the filtration $(\mathscr{F}_{n})_{n=0}^{\infty}$, where $\mathscr{F}_{n}$ are $\sigma$-sub-algebra generated by the dyadic intervals of length $2^{-n}$ in $[0,1]$. A double dyadic martingale is a double indexed random process $(f^{j}_{n})_{j,n}$ such that $(f^{j}_{n})_{j,n}$ is a dyadic martingale in $j$ for each fixed $n$ and also a dyadic martingale in $n$ for each fixed $j$. The following inequalities can be viewed as a variant of \cite[Theorem 1.1]{Ba} for Banach space-valued double dyadic martingales.

\begin{thm} \label{double dyadic martingale inequalities}
Suppose that $f^{j}=(f^{j}_{n})_{n=0}^{\infty}$ are dyadic martingales with values in $p$-uniformly smooth (resp. $q$-uniformly convex) Banach space for $1<p\leq 2$ (resp. $2\leq q<\infty$), $j=1,2,\cdots,m$. Then, there exists constant $K$ depending only on $X$ such that following inequalities hold
\begin{equation}\label{doble martingale type}
\left\|\big{(}\sum\limits_{j=1}^{m}\|f^{j}_{n}\|^{p}\big{)}^{1/p}\right\|_{r}\leq K r^{1/p}\left\|\big{(}\sum\limits_{j=1}^{m}S_{p,n}^{p}(f^{j})\big{)}^{1/p} \right\|_{r},
\end{equation}
\begin{equation}\label{double martingale cotype}
\left(\mbox{resp.}~ \left\|\big{(}\sum\limits_{j=1}^{m}S_{q,n}^{q}(f^{j})\big{)}^{1/q} \right\|_{r}\leq K r^{1/q} \left\|\big{(}\sum\limits_{j=1}^{m}\|f^{j}_{n}\|^{q}\big{)}^{1/q} \right\|_{r},\right)
\end{equation}
for all $2\leq r<\infty$ and $n\in\N$.
\end{thm}

To prove the theorem we will first deal with Hilbert space-valued dyadic martingales and then apply the ``martingale dimension reduction'' argument to reduce martingales from $p$-uniformly smooth space-valued to $\R^{2}$-valued. The following lemma is a variant of Theorem \ref{Hilbert valued martingale concentration inequality}.
\begin{lem}\label{Hilbert space-valued double dyadic martingale inequalities}
Suppose that $f^{j}=(f^{j}_{n})_{n=0}^{\infty}$ are dyadic martingales with values in Hilbert space $\HS$, for $j=1,2,\cdots,m$. Then, for every $r\geq 0$
\begin{equation}\nonumber
\Pb\left\{\Big{(}\sum\limits_{j=1}^{m}|f^{j}_{n}|^{2}\Big{)}^{1/2}\geq r \right\}\leq 2\exp\left\{-\frac{r^{2}}{2\Big{\|}\sum_{j=1}^{m}S_{2,n}^{2}(f^{j})\Big{\|}_{\infty}} \right\},
\end{equation}
for all $n\in \N$.
\end{lem}
\begin{proof}
Let $u(t)=\big{(}\sum_{j=1}^{m}|f^{j}_{n-1}+td_{n,f^{j}}|^{2}\big{)}^{1/2}$ and $\varphi(t)=\E_{n}\big{(}\cosh\lambda u(t)\big{)}$. Note that $f^{j}=(f^{j}_{n})_{n=0}^{\infty}$ are dyadic martingales for all $j=1,\cdots,m$, and hence $\left|d_{n,f^{j}}\right|$ is $\mathscr{F}_{n-1}$-measurable for each $n\in \N$ and $j=1,2\cdots,m$.

By the proof of Theorem \ref{Hilbert valued martingale concentration inequality}, we have
$\varphi(1)\leq \exp\left\{\frac{\lambda^{2}\sum_{j=1}^{m}\left|d_{n,f^{j}}\right|^{2}}{2} \right\}\varphi(0)$,
then it follows that $(g_{n})_{n=0}^{\infty}$ forms a super-martingale, where
\begin{equation}\nonumber
g_{n}=\frac{\cosh\lambda \big{(}\sum\limits_{j=1}^{m}|f^{j}_{n}|^{2}\big{)}^{1/2}}{\prod\limits_{j=1}^{m}\exp\Big{\{}\frac{\lambda^{2}S_{2,n}^{2}(f^{j})}{2}\Big{\}}},~~n\in\N_{0},~\lambda>0.
\end{equation}
Consequently, the result follows from the Markov inequality and the fact that $\int g_{n} d\Pb\leq\int g_{0} d\Pb=1$.
\end{proof}

\begin{proof}[Proof of Theorem \ref{double dyadic martingale inequalities}]
{\bf Step \uppercase\expandafter{\romannumeral1}.} Suppose that $f^{j}=(f^{j}_{n})_{n=0}^{\infty}$, $j=1,\cdots,m$ are dyadic martingales  with values in $p$-uniformly smooth space, then, by the ``martingale dimension reduction'' argument (Lemma \ref{martingale dimension reduction lemma}), it follows that there exists $\R^{2}$-valued martingales $N^{j}=(N^{j}_{n})_{n=0}^{\infty}$ such that
\begin{equation}\label{domain of martingales}
\|f^{j}_{n}\|^{p}\leq |N^{j}_{n}|^{2},
\end{equation}
and
\begin{equation}\label{domain of martingale differences}
|N^{j}_{n}-N^{j}_{n-1}|^{2}\leq K_{1}\|d_{n,f^{j}}\|^{p},
\end{equation}
$j=1,\cdots,m$ with $K_{1}=O(s_{p}(X))$.

Hence, analogous to the proof of Theorem \ref{p-uniform smoothness is necessary}, we can deduce the following inequality by using Lemma \ref{good lambda inequality} and Lemma \ref{Hilbert space-valued double dyadic martingale inequalities},
\begin{equation}\nonumber
\left\|\Big{(}\sum\limits_{j=1}^{m}\|f^{j}_{n}\|^{p}\Big{)}^{1/2} \right\|_{\gamma}^{\gamma}\leq (\gamma K_{1})^{\gamma/2}\left\|\Big{(}\sum\limits_{j=1}^{m}S_{p,n}^{p}(f^{j})\Big{)}^{1/2} \right\|_{\gamma}^{\gamma},
\end{equation}
for all $1\leq \gamma<\infty$. Therefore, the following holds
\begin{equation}\label{change variable}
\int \Big{(}\sum\limits_{j=1}^{m}\|f^{j}_{n}\|^{p} \Big{)}^{\delta} d\Pb\leq (2\delta K_{1})^{\delta} \int \Big{(}\sum\limits_{j=1}^{m} S_{p,n}^{p}(f^{j})\Big{)}^{\delta} d\Pb,
\end{equation}
for $\frac{1}{2}\leq \delta=\frac{\gamma}{2}<\infty$. Note that $1<p\leq 2$, hence, by \eqref{change variable}, the following inequality holds true for all $1\leq r<\infty$,
\begin{equation}\nonumber
\int\Big{(}\sum\limits_{j=1}^{m}\|f^{j}_{n}\|^{p} \Big{)}^{r/p} d\Pb\leq \left(\frac{2rK_{1}}{p}\right)^{r/p}\int \Big{(}\sum\limits_{j=1}^{m}S_{p,n}^{p}(f^{j})\Big{)}^{r/p} d\Pb.
\end{equation}
Consequently
\begin{equation}\nonumber
\left\|\Big{(}\sum\limits_{j=1}^{m}\|f^{j}_{n}\|^{p}\Big{)}^{1/p} \right\|_{r}\leq (2K_{1})^{1/p} r^{1/p} \left\|\Big{(}\sum\limits_{j=1}^{m}S_{p,n}^{p}(f^{j})\Big{)}^{1/p} \right\|_{r},
\end{equation}
for all $1\leq r<\infty$.

{\bf Step \uppercase\expandafter{\romannumeral2}.} We prove the inequality \eqref{double martingale cotype} via duality argument. For $2\leq q,~r<\infty$, denote that $r^{\prime}=\frac{r}{r-1}$ and $q^{\prime}=\frac{q}{q-1}$, then, for any $0<\varepsilon<1$, there exists a sequence $(\varphi_{l,j})_{j=1,l=1}^{m,n}\in L^{r^{\prime}}(X^*)$ such that
\begin{equation}\label{existence of sequence}
\left\|\Big{(}\sum\limits_{j=1}^{m}\sum\limits_{l=1}^{n} \|\varphi_{l,j}\|^{q^{\prime}}\Big{)}^{1/q^{\prime}} \right\|_{r^{\prime}}\leq 1,
\end{equation}
and
\begin{equation}\label{estimate by duality}
\begin{split}
(1-\varepsilon)\left\|\Big{(}\sum\limits_{j=1}^{m}S_{q,n}^{q}(f^{j})\Big{)}^{1/q} \right\|_{r}
&\leq \int\sum\limits_{j=1}^{m}\sum\limits_{l=1}^{n}\langle d_{l,f^{j}},\varphi_{l,j}\rangle d\Pb\\
&=\int\sum\limits_{j=1}^{m}\sum\limits_{l=1}^{n}\langle d_{l,f^{j}},(\E_{l}-\E_{l-1})\varphi_{l,j}\rangle d\Pb\\
&=\int \sum\limits_{j=1}^{m}\langle f^{j}_{n},g^{j}_{n}\rangle d\Pb,
\end{split}
\end{equation}
where $g^{j}_{n}=\sum_{l=1}^{n}(\E_{l}-\E_{l-1})\varphi_{l,j}$ for all $j=1,\cdots,m$.
We now estimate the term $\int \sum_{j=1}^{m}\langle f^{j}_{n},g^{j}_{n}\rangle d\Pb$ by a duality argument. Indeed,
if $X$ is of $q$-uniformly convex then $X^*$ must be $q^{\prime}$-uniformly smooth and hence, by the first part of our proof and the H\"older inequality, we have
\begin{equation}\label{Pisier's interpolation argument}
\begin{split}
\int\sum\limits_{j=1}^{m}\langle f^{j}_{n},g^{j}_{n}\rangle d\Pb&\leq\int\sum\limits_{j=1}^{m}\|f^{j}_{n}\|\cdot\|g^{j}_{n}\| d\Pb\\
                                                                &\leq\int\Big{(}\sum\limits_{j=1}^{m}\|f^{j}_{n}\|^{q}\Big{)}^{1/q}\cdot\Big{(}\sum\limits_{j=1}^{m}\|g^{j}_{n}\|^{q^{\prime}}\Big{)}^{1/q^{\prime}}d\Pb\\
                                                                &\leq\left\|\Big{(}\sum\limits_{j=1}^{m}\|f^{j}_{n}\|^{q}\Big{)}^{1/q}\right\|_{r}\cdot\left\|\Big{(}\sum\limits_{j=1}^{m}\|g^{j}_{n}\|^{q^{\prime}}\Big{)}^{1/q^{\prime}}\right\|_{r^{\prime}}\\
                                                                &\leq 2r^{\prime}(r^{\prime}\!-\!1)^{-\!1/q}\left\|\Big{(}\sum\limits_{j=1}^{m}\|f^{j}_{n}\|^{q}\Big{)}^{1/q}\right\|_{r}\!\!\cdot\left\|\!\Big{(}\!\sum\limits_{j=1}^{m}\sum\limits_{l=1}^{n} \|\varphi_{l,j}\|^{q^{\prime}}\!\Big{)}^{1/q^{\prime}}\! \right\|_{r^{\prime}}\!\!,
\end{split}
\end{equation}
where the last inequality follows from \cite[p. 418]{Pis2} and the triangle inequality.

Combining with \eqref{existence of sequence}, \eqref{estimate by duality}, \eqref{Pisier's interpolation argument} and note that $0<\varepsilon<1$ is arbitrary, then there exists a constant $K_{2}$ only depending on $X$ such that
\begin{equation}\nonumber
\left\|\Big{(}\sum\limits_{j=1}^{m}S_{q,n}^{q}(f^{j})\Big{)}^{1/q} \right\|_{r}\leq K_{2} r^{\prime(1+1/q^{\prime})}(r^{\prime}-1)^{-1/q}\left\|\Big{(}\sum\limits_{j=1}^{m}\|f^{j}_{n}\|^{q}\Big{)}^{1/q}\right\|_{r}.
\end{equation}
Note here that $2\leq r<\infty$, then it entails $r^{\prime(1+1/q^{\prime})}(r^{\prime}-1)^{-1/q}\leq 4 r^{1/q}$. Therefore, letting $K=4K_{2}$ we have
\begin{equation}\nonumber
\left\|\Big{(}\sum\limits_{j=1}^{m}\|f^{j}_{n}\|^{q}\Big{)}^{1/q}\right\|_{r}\leq K r^{1/q} \left\|\Big{(}\sum\limits_{j=1}^{m}S_{q,n}^{q}(f^{j})\Big{)}^{1/q}\right\|_{r},
\end{equation}
for all $2\leq r<\infty$. This completes our proof.
\end{proof}

\begin{cor}
A Banach space $X$ is linearly isomorphic to a Hilbert space if and only if the following two inequalities hold for $X$-valued dyadic martingales $f^{j}=(f^{j}_{n})_{n=0}^{\infty}$, $j=1,\cdots,m$, there exists a constant $K$ only depending on $X$ such that
\begin{equation}\label{double martingale type 2}
\left\|\Big{(}\sum\limits_{j=1}^{m}\|f^{j}_{n}\|^{2}\Big{)}^{1/2} \right\|_{r}\leq K r^{1/2}\left\|\Big{(}\sum\limits_{j=1}^{m}S_{2,n}^{2}(f^{j})\Big{)}^{1/2} \right\|_{r},
\end{equation}
and
\begin{equation}\label{double martingale cotype 2}
\left\|\Big{(}\sum\limits_{j=1}^{m}S_{2,n}^{2}(f^{j})\Big{)}^{1/2}\right\|_{r}\leq K r^{1/2}\left\|\Big{(}\sum\limits_{j=1}^{m}\|f^{j}_{n}\|^{2}\Big{)} \right\|_{r},
\end{equation}
for all $2\leq r<\infty$.
\end{cor}
\begin{proof}
By the fact that Hilbert space is $2$-uniformly smooth and $2$-uniformly convex, then \eqref{double martingale type 2} and
\eqref{double martingale cotype 2} follow from Theorem \ref{double dyadic martingale inequalities} directly. Hence, the ``only if'' part is proven.

For the ``if'' part, we assume that inequalities \eqref{double martingale type 2} and \eqref{double martingale cotype 2} hold true for $X$-valued dyadic martingales. Then, $X$ must have martingale type $2$ and martingale cotype $2$. By a theorem of Kwapie\'{n}, it follows that $X$ is linearly isomorphic to a Hilbert space.
\end{proof}

\subsection{De la Pe\~{n}a-type inequalities for $p$-uniformly smooth space-valued martingales}
We conclude this paper by extending the De la Pe\~{n}a-type concentration inequalities for martingales with values in $p$-uniformly smooth space.

Recall that the following martingale inequality is known as the Freedman inequality, which was proved by Freedman \cite{Fr} in 1975.
\begin{thm}[Freedman inequality]\label{Freedman inequality}
For every real-valued martingale $f=(f_{j})_{j=0}^{\infty}$ with $|d_{j,f}|\leq 1~a.s.$ for all $j\in\N$ and any positive numbers $a,~b$. The following inequality holds
\begin{equation}\nonumber
\Pb\{f_{n}-f_{0}\geq r,~s_{2,n}^{2}(f)\leq b~\mbox{for some }n\}\leq\Big{(}\frac{b}{r+b}\Big{)}^{r+b}e^{r}\leq\exp\Big{\{}-\frac{r^{2}}{2(r+b)}\Big{\}},
\end{equation}
where $s_{2,n}^{2}(f)=\sum_{j=1}^{n}\E_{j-1}|d_{j,f}|^{2}$.
\end{thm}
In order to avoid any boundness assumption on martingale differences, De la Pe\~{n}a \cite{Vi} proved the following inequality for conditionally symmetric martingales.

\begin{thm}[De la Pe\~{n}a]\label{delapena inequality 1}
Let $f=(f_{j})_{j=0}^{\infty}$ a real-valued conditionally symmetric martingale. Then for $r,~b>0$, we have
\begin{equation}\nonumber
\Pb\{f_{n}-f_{0}\geq r,~S_{2,n}^{2}(f)\leq b~\mbox{for some }n\}\leq \exp\Big{\{}-\frac{r^{2}}{2b} \Big{\}}.
\end{equation}
\end{thm}
\begin{thm}[De la Pe\~{n}a]\label{delapena inequality 2}
Let $f=(f_{j})_{j=0}^{\infty}$ be a real-valued conditionally symmetric martingale. Then, for $\beta,~r,~b>0$ and $\alpha\geq 0$, the following inequality holds
\begin{equation}\nonumber
\Pb\left\{\frac{f_{n}-f_{0}}{\alpha+\beta S_{2,n}^{2}(f)}\geq r,~\frac{1}{S_{2,n}^{2}(f)}\leq b~\mbox{for some }n \right\}\leq \exp\left\{-r^{2}\big{(}\frac{\beta^{2}}{2b}+\alpha\beta\big{)} \right\}.
\end{equation}
\end{thm}

By an argument from Pinelis \cite{Pin2} one can extend the Freedman inequality (Theorem \ref{Freedman inequality}) and the De la Pe\~{n}a inequality (Theorem \ref{delapena inequality 1}) to $2$-uniformly smooth space valued martingales. Moreover, by applying the ``martingale dimension reduction'' argument (Lemma \ref{martingale dimension reduction lemma}) one can further derive that Theorem \ref{Freedman inequality} and Theorem \ref{delapena inequality 1} hold true for martingales with values in $p$-uniformly smooth space by replacing $S_{2,n}^{2}(f)$ to $S_{p,n}^{p}(f)$ of the left hand side and $r$, $b$ to $r^{p/2}$, $Kb$ of the right hand side respectively with $K=O(s_{p}(X))$.

We now conclude this paper by extending the De la Pe\~{n}a inequality (Theorem \ref{delapena inequality 2}) to $p$-uniformly smooth space-valued conditionally symmetric martingales.

\begin{thm}
Let $f=(f_{j})_{j=0}^{\infty}$ be an $X$-valued conditionally symmetric martingale \big{(}relatively to the filtration $(\mathscr{F}_{j})_{j=0}^{\infty}$\big{)}, where $X$ is a $p$-uniformly smooth space ($1<p\leq 2$). Then, for each $\beta,~r,~b>0$ and $\alpha\geq 0$ the following inequality holds true,
\begin{equation}\label{p-delapena inequality}
\Pb\left\{\frac{\|f_{n}-f_{0}\|}{\big{(}\!\alpha\!+\!\beta S_{p,n}^{p}(f)\!\big{)}^{2/p}}\!\geq\! r,~\frac{1}{S_{p,n}^{p}(f)}\!\leq b~\mbox{for some }n \right\}\!\leq\! 4\!\exp\!
\left\{\!-\!r^{p}\big{(}\!\frac{\beta^{2}}{2bK}\!+\!\frac{\alpha\beta}{K}\!\big{)}\!\right\},
\end{equation}
where $K=O(s_{p}(X))$.
\end{thm}

\begin{proof}
We prove the inequality for $\HS$-valued conditionally symmetric martingales firstly and then apply the ``martingale dimension reduction'' argument (Lemma \ref{martingale dimension reduction lemma}) to prove the case for $p$-uniformly smooth space-valued martingales.

Without loss of generality, we assume that $f_{0}=0~a.s.$ throughout the following proof.

{\bf Step {\uppercase\expandafter{\romannumeral1}}.} Suppose that $f=(f_{j})_{j=0}^{\infty}$ is a $\HS$-valued conditionally symmetric martingale \big{(}relatively to the filtration $(\mathscr{F}_{j})_{j=0}^{\infty}$\big{)}, then applying
Proposition \ref{Burkholder's replace filtration argument} and \eqref{super-martingale} of Theorem \ref{Hilbert valued martingale concentration inequality}, there exists another filtration $(\mathscr{G}_{j})_{j=0}^{\infty}$ such that $f=(f_{j})_{j=0}^{\infty}$ is a martingale and $g=(g_{j})_{j=0}^{\infty}$ is a nonnegative super-martingale relatively to the filtration $(\mathscr{G}_{j})_{j=0}^{\infty}$, where
\begin{equation}\nonumber
g_{n}=\frac{\cosh\lambda|f_{n}|}{\prod\limits_{j=1}^{n}\exp\{\frac{\lambda^{2}|d_{j,f}|^{2}}{2}\}},~~n\in\N_{0}~\mbox{and}~\lambda>0,
\end{equation}

Note that Hilbert space is $2$-uniformly smooth, hence we now prove inequality \eqref{p-delapena inequality} for $p=2$. Indeed, let $B=\left\{\frac{|f_{n}|}{\alpha+\beta S_{2,n}^{2}(f)}\geq r,~\frac{1}{S_{2,n}^{2}(f)}\leq b~\mbox{for some }n \right\}$ and define $\sigma=\inf\left\{n\in\N_{0}:\frac{|f_{n}|}{\alpha+\beta S_{2,n}^{2}(f)}\geq r,~\frac{1}{S_{2,n}^{2}(f)}\leq b\right\}$. It is clear that $\sigma$ is a stopping time satisfying
\[
B=\left\{\frac{|f_{n}|}{\alpha+\beta S_{2,n}^{2}(f)}\geq r,~\frac{1}{S_{2,n}^{2}(f)}\leq b~\mbox{for some }n \right\}=\{\sigma<\infty\}.
\]
Then,
\begin{equation}\nonumber
\begin{split}
\Pb(B)&\leq 2\int \frac{\cosh\frac{\lambda}{2}|f_{\sigma}|}{\exp\big{\{}\frac{\lambda r}{2}(\alpha+\beta S_{2,\sigma}^{2}(f))\big{\}}}\cdot\mathds{1}_{\{\sigma<\infty\}}d\Pb\\
          &=2\! \int\!\frac{\cosh\frac{\lambda}{2}|f_{\sigma}|}{\prod\limits_{j=0}^{n}\!\exp\big{\{}\!\frac{\lambda^{2}|d_{\sigma,f}|^{2}}{4}\!\big{\}}}\!\cdot\!\exp\!\left\{\!\frac{\lambda^{2}S_{2,\sigma}^{2}(f)}{4}\!-\!\frac{\lambda r}{2}(\alpha\!+\!\beta S_{s,\sigma}^{2}(f))\!\right\}\!\cdot\!\mathds{1}_{\{\sigma<\infty\}} d\Pb,
\end{split}
\end{equation}
for all $\lambda>0$.

By the Cauchy-Schwarz inequality and the fact that $(g_{j})_{j=0}^{\infty}$ is a nonnegative super-martingale with $\E(g_{0})=1$, we have
\begin{equation}\nonumber
\begin{split}
\Pb(B)&\leq 2\inf\limits_{\lambda>0}\left(\int \exp\big{\{}\frac{\lambda^{2}S_{2,\sigma}^{2}(f)}{2}-(\alpha+\beta S_{2,\sigma}^{2}(f))\lambda r\big{\}} \cdot\mathds{1}_{B}d \Pb \right)^{\frac{1}{2}}\\
          &\leq 2 \left(\int \exp\big{\{}-\frac{r^{2}(\alpha+\beta S_{2,\sigma}^{2}(f))^{2}}{2 S_{2,n}^{2}(f)}\big{\}}\cdot\mathds{1}_{B}d\Pb \right)^{\frac{1}{2}}\\
          &\leq 2 \left(\int\exp\big{\{}-\alpha\beta r^{2}-\frac{r^{2}\beta^{2}S_{2,\sigma}^{2}(f)}{2} \big{\}}\cdot\mathds{1}_{B}d\Pb \right)^{\frac{1}{2}}\\
          &\leq 2 \left\{\exp-r^{2}(\alpha\beta+\frac{\beta^{2}}{2b}) \right\}^{\frac{1}{2}}\cdot \Pb(B)^{\frac{1}{2}}.
\end{split}
\end{equation}
Therefore,
\[
\begin{split}
\Pb(B)&=\Pb\left\{\frac{|f_{n}|}{\alpha+\beta S_{2,n}^{2}(f)}\geq r,~\frac{1}{S_{2,n}^{2}(f)}\leq b~\mbox{for some }n \right\}\\
      &\leq 4\exp\left\{-r^{2}(\frac{\beta^{2}}{2b}+\alpha\beta) \right\}.
\end{split}
\]

{\bf Step {\uppercase\expandafter{\romannumeral2}}.} We now turn to the case when $f=(f_{j})_{j=0}^{\infty}$ is a conditionally symmetric martingale with values in $p$-uniformly smooth space. By Lemma \ref{martingale dimension reduction lemma} there exists an $\R^{2}$-valued martingale $N=(N_{j})_{j=0}^{\infty}$ (relatively to another filtration) such that
\begin{enumerate}[i)]
\item $\|f_{j}\|^{p}\leq |N_{j}|^{2}$,
\item $|N_{j}-N_{j-1}|^{2}\leq K_{1}\|d_{j,f}\|^{p}$,
\end{enumerate}
for all $j\in\N$ with $K_{1}=O(s_{p}(X))$.

Furthermore, by the definition of $(N_{j})_{j=0}^{\infty}$ (see \eqref{definition of martingale}), we have for $j\in \N$
\begin{equation}\label{addition inequality}
\begin{split}
|N_{j}-N_{j-1}|^{2}&=\frac{p^{2}}{4}\cdot\frac{\langle J_{p}(f_{j-1}),f_{j}-f_{j-1}\rangle^{2}\cdot\mathds{1}_{A_{j-1}}}{|N_{j-1}|^{2}}+(c+p)\|f_{j}-f_{j-1}\|^{p}\\
                         &\geq (c+p)\|f_{j}-f_{j-1}\|^{p}-\frac{p^{2}}{4}\cdot\frac{\|f_{j}-f_{j-1}\|^{2}\cdot \mathds{1}_{A_{j-1}}}{\big{\|}N_{j-1}\big{\|}_{2}^{2(1-\frac{2}{p})}}\\
                         &\geq (c+p-\frac{p^{2}}{4})\|f_{j}-f_{j-1}\|^{p}\\
                         &\geq c\|f_{j}-f_{j-1}\|^{p},
\end{split}
\end{equation}
where $c=O(s_{p}(X))$. The last inequality of \eqref{addition inequality} follows from the fact that $1<p\leq 2$.

Then, by using $i)$, $ii)$ and \eqref{addition inequality}, we have
\begin{equation}\nonumber
\begin{split}
                    &\Pb\left\{\frac{\|f_{n}\|}{\big{(}\alpha+\beta S_{p,n}^{p}(f)\big{)}^{2/p}}\geq r,~\frac{1}{S_{p,n}^{p}(f)}\leq b~\mbox{for some }n \right\}\\
                  =&\Pb\left\{\frac{\|f_{n}\|^{\frac{p}{2}}}{\alpha+\frac{\beta}{K_{1}}K_{1} S_{p,n}^{p}(f)}\geq r^{\frac{p}{2}},~\frac{1}{c S_{p,n}^{p}(f)}\leq\frac{b}{c}~\mbox{for some }n \right\}\\
               \leq&\Pb\left\{\frac{|N_{n}|}{\alpha+\frac{\beta}{K_{1}}S_{2,n}^{2}(N)}\geq r^{\frac{p}{2}},~\frac{1}{S_{2,n}^{2}(N)}\leq\frac{b}{c}~\mbox{for some }n \right\}\\
               \leq& 4\exp\left\{-r^{p}\big{(}\frac{\beta^{2}}{2Kb}+\frac{\alpha\beta}{K} \big{)} \right\},
\end{split}
\end{equation}
where $K=O(s_{p}(X))$.
\end{proof}

\section*{Acknowledgements}
The author is partially supported by Natural Science Foundation of China (Grant No. 12071240). I would like to thank the Yau Mathematical Sciences Center of Tsinghua University for providing a great opportunity for me to complete this work. And I am grateful to the Associate Editor and the referees for their consideration.

\end{document}